\documentclass{article}
\usepackage{amsmath}
\usepackage{amsfonts}
\usepackage{amsthm}
\usepackage{amssymb}
\usepackage{enumerate}

\newtheorem{thm}{Theorem}[section]
\newtheorem{prop}[thm]{Proposition}
\newtheorem{cor}[thm]{Corollary}

\newtheorem{lem}[thm]{Lemma}

\theoremstyle{definition}

\newtheorem{rmk}[thm]{Remark}

\begin{document}

\title{Integral Bases and Invariant Vectors for Weil Representations}

\author{Shaul Zemel}

\maketitle

\section*{Introduction}

Let $D$ be a discriminant form, also known as finite quadratic module. Then there is a representation $\rho_{D}$ of $\operatorname{SL}_{2}(\mathbb{Z})$ (or sometimes its double cover), called the \emph{Weil representation} associated with $D$ on the space $\mathbb{C}[D]$. It is defined on the generators of that group by Equation \eqref{Weildef}. These representations are essentially the case of the finite group $D$ in the general theory of Weil representations, initiated in \cite{[We]}.

These Weil representations now form an important technical tool in the theory of modular forms. Indeed, when the discriminant form $D$ comes from an even lattice $L$, the most natural way to present the theta function of $L$ is as a vector-valued modular form of level 1 and the Weil representation associated with $D$. This allows for many modularity proofs to be reduced to the case of level 1, where the group is $\operatorname{SL}_{2}(\mathbb{Z})$ (or its double cover), with well-known generators and relations.

Given a discriminant form $D$, the subspace of $\mathbb{C}[D]$ that is invariant under $\rho_{D}$ is of particular importance. Indeed, it is the space of holomorphic modular forms (constants) of weight 0 with representation $\rho_{D}$. As one example of the role it plays, assume that $D$ is the discriminant form of a lattice $L$ of signature $(2,2)$. Then a weakly holomorphic modular form of weight 0 is determined by its singular part precisely up to this subspace. Thus knowing it resolves certain technical questions involving the corresponding theta lifts.  

We shall not give a comprehensive list of references involving these questions. The latter arose, for example, in \cite{[BZ]}, a reference which cites many papers discussing various forms of theta lifts. However, we do mention another result, illustrating the importance of Weil representations: The paper \cite{[NW]} shows that every irreducible representation of $\operatorname{SL}_{2}(\mathbb{Z})$ that factors through a congruence quotient is contained inside some Weil representation of this sort.

\smallskip

A seemingly unrelated question involving Weil representations is the following one. Equation \eqref{Weildef} easily implies that in the natural basis for $\mathbb{C}[D]$, the Weil representation $\rho_{D}$ is defined over an explicit cyclotomic number field. However, the coefficients include denominators. In the 33rd Automorphic Forms Workshop, L. Candelori presented the question of finding one may ask whether there is a basis for $\mathbb{C}[D]$ in which the coefficients of the Weil representation are algebraic \emph{integers}, and presented some initial results in the case of discriminant forms of prime order. These, in fact, turn out to be the \emph{hardest} part of this question.

The main result of this paper (Theorem \ref{fin} below) solves this question entirely. Moreover, large parts of a general representation $\rho_{D}$ consists of subspaces admitting bases on which the action of every element on every basis element yields another basis element multiplied by a root of unity. By completing this with the result about prime discriminants (and some simple cases over the prime 2), the result for a general discriminant form follows.

\smallskip

For explaining the basic idea, note that the inversion element $S\in\operatorname{SL}_{2}(\mathbb{Z})$ takes, under $\rho_{D}$, a natural basis element $\mathfrak{e}_{\gamma}$ of $\mathbb{C}[D]$ (with $\gamma \in D$) to a certain sum over all the $\mathfrak{e}_{\delta}$'s, with coefficients that are based on $\gamma$. By taking a subgroup $H$ of $D$, one defines in Equation \eqref{vecH} intermediate vectors between these two types of vectors. For the correct choice of $H$, these vectors produce bases with the properties that we seek.

As an example, assume that $H$ is a self-dual isotropic subgroup of $D$. In this case Theorem \ref{genact} shows that any element of $\operatorname{SL}_{2}(\mathbb{Z})$ takes such a vector associated with $H$ to another such vector associated with $H$, times an explicit root of unity. Moreover, the action on the indices of these vectors is just the action of $\operatorname{SL}_{2}(\mathbb{Z})$ on vectors of length 2 in $D$.

Assume now that $H$ is isotropic, yielding a quotient of exponent $p$, and let $J \subseteq H$ be of size $p$, with quotient $B$ (of size $|D|/p^{2}$). In this case $\mathbb{C}[B]$ embeds into $\mathbb{C}[D]$ as a sub-representation via the arrow operator from \cite{[Br]} and others, and Theorem \ref{arrowJperp} shows (via Remark \ref{Jperpiso}) that its orthogonal complement inside $\mathbb{C}[D]$ also admits a basis with a similar action (but up to some 8th root of unity that are harder to determine). This reduces the question to anisotropic discriminant forms, whose $p$-parts are given in, e.g., \cite{[Zh]}, and solving these cases yields the final result. We remark that in the prime discriminant case we do not obtain the formula for a general element in this basis.

Finally, the fact that in large enough parts of the representations we have a closed formula for the action of a general element allows one to use the classical formula of Frobenius to determine the dimension of the space of invariants in $\mathbb{C}[D]$ discussed above. Note that some properties of this space are known in general: In relation to our integrality question, the result from \cite{[ES]} proves that $\mathbb{C}[D]^{\mathrm{inv}}$ is defined over $\mathbb{Z}$. Moreover, any self-dual isotropic subgroup yields a 1-dimensional subspace of $\mathbb{C}[D]$, and a theorem of \cite{[NRS]} (also proven in \cite{[Bi]}) shows that when such subgroups exist, these spaces generate $\mathbb{C}[D]^{\mathrm{inv}}$. While determining this dimension in this method might still present difficulties in general, we obtain the formula for $\dim\mathbb{C}[D]^{\mathrm{inv}}$ for generalized hyperbolic planes in Theorem \ref{hypplane}, and for discriminants of prime level in Theorem \ref{Fpvs}.

\smallskip

I am grateful to L. Candelori for presenting this question at the AFW, as well as for commenting some early drafts. Special thanks are due to Y. Wang, for observing in these drafts the pattern lying behind the result of Theorem \ref{genact}. I thank J. Bruinier, N. Scheithauer, S. Ehlen, and P. Kiefer for discussions on this subject, and in particular to P. Bieker for sharing the content of \cite{[Bi]} with me. The idea of E. de-Shalit of using the decomposition from Lemma \ref{Vdm} is also gratefully acknowledged. Finally, I would like to express my gratitude to the hospitality of Duquesne university for the 33rd AFW, and to the organizers there.

\smallskip

The paper is divided into 5 sections. Section \ref{Subgp} introduces the Weil representations and the vectors that we later use for our bases. Section \ref{TrivQuot} proves the formula for self-dual isotropic and quasi-isotropic subgroups, and Section \ref{TrivQuot} establishes the result involving quotients of prime level. Then Section \ref{GenRes} produces the construction of integral bases in general, and finally Section \ref{InvVect} shows how to apply these formulae to the determination of the dimension of the space of invariant vectors.

\section{Weil Representations and Subgroups \label{Subgp}}

The group $\operatorname{SL}_{2}(\mathbb{Z})$ is known to be generated by the elements $T=\big(\begin{smallmatrix} 1 & 1 \\ 0 & 1\end{smallmatrix}\big)$ and $S=\big(\begin{smallmatrix} 0 & -1 \\ 1 & 0\end{smallmatrix}\big)$, whose only relations are $S^{2}=(ST)^{3}=Z$ and $Z^{2}=I$ (the matrix $Z$ is minus the identity matrix). It admits a non-trivial double cover, denoted by $\operatorname{Mp}_{2}(\mathbb{Z})$, which is generated by appropriate lifts of $T$ and $S$, satisfying the first relation but in which $Z$ now has order 4 (and $Z^{2}$ generates the kernel of the projection onto $\operatorname{SL}_{2}(\mathbb{Z})$, of order 2).

Let $D$ be a \emph{discriminant form}, also called a \emph{finite quadratic module}, namely a finite group with a $\mathbb{Q}/\mathbb{Z}$-valued quadratic form, which we write as $\gamma\mapsto\frac{\gamma^{2}}{2}$. It induces the $\mathbb{Q}/\mathbb{Z}$-valued bilinear form $(\gamma,\delta)=\frac{(\gamma+\delta)^{2}}{2}-\frac{\gamma^{2}}{2}-\frac{\delta^{2}}{2}$, which we assume to be non-degenerate. This identified $D$ with its $\mathbb{Q}/\mathbb{Z}$-dual, and there is a signature $\operatorname{sgn}D\in\mathbb{Z}/8\mathbb{Z}$ which can be defined using \emph{Milgram's formula}
\begin{equation}
\sum_{\gamma \in D}\mathbf{e}\big(\tfrac{\gamma^{2}}{2}\big)=\mathbf{e}(\operatorname{sgn}D/8)\cdot\sqrt{|D|}. \label{Milgram}
\end{equation}

To $D$ we associate the \emph{Weil representation} $\rho_{D}$ of $\operatorname{Mp}_{2}(\mathbb{Z})$ onto the space $\mathbb{C}[D]$, with the natural basis $\{\mathfrak{e}_{\gamma}\}_{\gamma \in D}$, via the formulae
\begin{equation}
\rho_{D}(T)\mathfrak{e}_{\gamma}=\mathbf{e}\big(\tfrac{\gamma^{2}}{2}\big)\mathfrak{e}_{\gamma}\qquad\mathrm{and}\qquad\rho_{D}(S)\mathfrak{e}_{\gamma}=\tfrac{\mathbf{e}(-\operatorname{sgn}D/8)}{\sqrt{|D|}}\sum_{\delta \in D}\mathbf{e}\big(-(\gamma,\delta)\big)\mathfrak{e}_{\delta}, \label{Weildef}
\end{equation}
where $\mathbf{e}(x)$ stands for $e^{2\pi ix}$. The fact that it is a representation can be proved using Milgram's formula, and it factors through a representation of $\operatorname{SL}_{2}(\mathbb{Z})$ if and only if $\operatorname{sgn}D$ is even (otherwise $Z^{2}$ acts as $-1$). We endow $\mathbb{C}[D]$ with the inner product in which the basis $\{\mathfrak{e}_{\gamma}\}_{\gamma \in D}$ is orthonormal, and then $\rho_{D}$ becomes a unitary representation.

\smallskip

Given any subgroup $H$ of $D$, the subgroup \[H^{\perp}=\big\{\gamma \in D\big|\;(\gamma,\delta)=0\;\forall\delta \in H\big\}\] has index $|H|$ in $D$, and the $\mathbb{Q}/\mathbb{Z}$-dual of $H$ is $D/H^{\perp}$. For such $H$ and elements $\eta$ and $\lambda$ in $D$ we define
\begin{equation}
\mathfrak{a}_{\eta,\lambda}^{H}:=\frac{1}{\sqrt{|H|}}\sum_{\gamma \in H}\mathbf{e}\big((\gamma,\eta)\big)\mathfrak{e}_{\lambda+\gamma}. \label{vecH}
\end{equation}
Their first properties are given in the following lemma.
\begin{lem}
The vector from Equation \eqref{vecH} depends only on the image of $\eta$ modulo $H^{\perp}$, and adding $\delta \in H$ to $\lambda$ multiplies $\mathfrak{a}_{\eta,\lambda}^{H}$ by $\mathbf{e}\big(-(\delta,\eta)\big)$. Choosing a set $\mathfrak{R}$ of representatives for $D/H$ in $D$, the set $\{\mathfrak{a}_{\eta,\lambda}^{H}\}_{\eta \in D/H^{\perp},\lambda\in\mathfrak{R}}$ is an orthonormal basis for $\mathbb{C}[D]$. \label{ortho}
\end{lem}

\begin{proof}
The dependence on $\eta$ in Equation \eqref{vecH} is by pairings with $H$, and adding $\delta \in H$ to $\lambda$ results in a summation index change. The pairing of $\mathfrak{a}_{\eta,\lambda}^{H}$ with $\mathfrak{a}_{\kappa,\nu}^{H}$ for $\nu\neq\lambda\in\mathfrak{R}$ (i.e., $\nu\not\in\lambda+H$) is based on disjoint subsets of $\{\mathfrak{e}_{\gamma}\}_{\gamma}$, and if $\nu=\lambda$ the pairing is $\frac{1}{|H|}\sum_{\gamma \in H}\mathbf{e}\big((\gamma,\eta-\kappa)\big)$. Since this equals 1 when $\kappa\in\eta+H^{\perp}$ and vanishes otherwise, the orthonormality follows. This proves the lemma.
\end{proof}
In particular, if $\eta \in H^{\perp}$ then $\mathfrak{a}_{\eta,\lambda}^{H}$ is invariant under replacing $\lambda$ by $\lambda+\delta$ with $\delta \in H$.
\begin{rmk}
Each basis vector $\mathfrak{e}_{\gamma}$ can be written as $\mathfrak{a}_{\eta,\gamma}^{\{0\}}$ via Equation \eqref{vecH}, where the independence of $\eta \in D$ and the orthonormality are generalized in Lemma \ref{ortho}. Moreover, we can rewrite Equation \eqref{Weildef} as stating that $\rho_{D}(T)$ and $\rho_{D}(S)$ take $\mathfrak{e}_{\gamma}=\mathfrak{a}_{\eta,\gamma}^{\{0\}}$ to $\mathbf{e}\big(\frac{\lambda^{2}}{2}\big)\mathfrak{a}_{\eta+\gamma,\gamma}^{\{0\}}$ and $\mathbf{e}(-\operatorname{sgn}D/8)\mathfrak{a}_{-\gamma,0}^{D}$ respectively. We soon generalize these formulae in Lemma \ref{rhoDaH} below. \label{ecana0}
\end{rmk}

\smallskip

A subgroup $H$ of $D$ is called \emph{quasi-isotropic} if $(\gamma,\delta)=0$ for every $\gamma$ and $\delta$ in $H$, namely, if $H \subseteq H^{\perp}$, and \emph{isotropic} when $\frac{\gamma^{2}}{2}=0$ for every $\gamma \in H$. Any isotropic subgroup is quasi-isotropic, and the isotropic condition is equivalent to the quotient $A:=H^{\perp}/H$ inheriting from $D$ a natural structure of a discriminant form. This discriminant form is non-degenerate, and appropriately gathering elements in Milgram's formula from Equation \eqref{Milgram} shows that it has the same signature as $D$. The following lemma is an immediate consequence of the definition.
\begin{lem}
The subgroup $H$ is quasi-isotropic if and only if $\gamma\mapsto\frac{\gamma^{2}}{2}$ is linear on $H$, namely there is $\xi_{H} \in D$ such that $\frac{\gamma^{2}}{2}=(\gamma,\xi_{H})$ for every $\gamma \in H$. The vector $\xi_{H}$ is unique modulo $H^{\perp}$, and its image in $D/H^{\perp}$ is trivial if $H$ is isotropic and has order 2 otherwise. \label{xiHdef}
\end{lem}
We shall also use the following consequence of Lemma \ref{xiHdef}.
\begin{cor}
Every quasi-isotropic subgroup $H$ of $D$ contains a unique maximal isotropic subgroup $H_{0}$. It equals $H$ when $H$ is isotropic, and has index 2 there otherwise. The associated subgroup $H_{0}^{\perp}$ is the union of $H^{\perp}\cup(\xi_{H}+H^{\perp})$. \label{isoinqiso}
\end{cor}
Indeed, the subgroup $H_{0}$ from Corollary \ref{isoinqiso} is just the kernel of the pairing with $\xi_{H}$ from Lemma \ref{xiHdef}.

The fact that a $\mathbb{Q}/\mathbb{Z}$-valued quadratic form is linear if and only if the associated bilinear form vanishes, and thus the quadratic form takes values in $\frac{1}{2}\mathbb{Z}/\mathbb{Z}$, extends Lemma \ref{xiHdef} as follows.
\begin{lem}
If $H$ is a subgroup of $D$ and $l\in\mathbb{Z}$ is such that $(\gamma,\delta)\in\frac{1}{l}\mathbb{Z}/\mathbb{Z}$ for every $\gamma$ and $\delta$ in $H$ then there exists a vector $\xi_{l,H} \in D$, unique modulo $H^{\perp}$, such that $l\frac{\gamma^{2}}{2}=(\gamma,\xi_{l,H})$ for all $\gamma \in H$. \label{xilH}
\end{lem}
The vector $\xi_{l,H}$ from Lemma \ref{xilH} is closely related to the element denoted by $x_{c}$ in \cite{[Sch]}, \cite{[Str]}, and \cite{[Ze]}, with $c=l$.

\smallskip

Most of our calculations will be based on evaluating $\rho_{D}$ on the vectors from Equation \eqref{vecH}. On the generators we get the following result.
\begin{lem}
For every $\lambda$ and $\eta$ we have the equality \[\rho_{D}(S)\mathfrak{a}_{\eta,\lambda}^{H}=\mathbf{e}(-\operatorname{sgn}D/8)\mathbf{e}\big(-(\lambda,\eta)\big)\mathfrak{a}_{-\lambda,\eta}^{H^{\perp}}.\] If $l\in\mathbb{Z}$ is as in Lemma \ref{xilH} and $\xi_{l,H}$ is the resulting vector, then we get \[\rho_{D}(T^{l})\mathfrak{a}_{\eta,\lambda}^{H}=\mathbf{e}\big(l\tfrac{\lambda^{2}}{2}\big)\mathfrak{a}_{\eta+l\lambda+\xi_{l,H},\lambda}^{H}.\] \label{rhoDaH}
\end{lem}

\begin{proof}
Unfolding the definitions from Equations \eqref{Weildef} and \eqref{vecH} and recalling that $|H^{\perp}|=\frac{|D|}{|H|}$ expresses $\rho_{D}(S)\mathfrak{a}_{\eta,\lambda}^{H}$ as $\mathbf{e}(-\operatorname{sgn}D/8)\big/\sqrt{|H^{\perp}|}$ times \[\sum_{\gamma \in H}\sum_{\beta \in D}\frac{\mathbf{e}\big((\gamma,\eta)-(\lambda+\gamma,\beta)\big)}{|H|}\mathfrak{e}_{\beta}=\sum_{\beta \in D}\mathbf{e}\big(-(\lambda,\beta)\big)\delta_{\beta+H^{\perp},\eta+H^{\perp}}\mathfrak{e}_{\beta}.\] Then writing $\beta$ as $\eta+\sigma$ with $\sigma \in H^{\perp}$ and applying Equation \eqref{vecH} (with $H^{\perp}$) yields the first expression. For the second one, we the multiplier $\mathbf{e}\big(l\frac{(\lambda+\gamma)^{2}}{2}\big)$ in front of $\mathfrak{e}_{\lambda+\gamma}$ in $\rho_{D}(T^{l})\mathfrak{a}_{\eta,\lambda}^{H}$ expanded as in Equation \eqref{vecH} by $\mathbf{e}\big(l\frac{\lambda^{2}}{2}\big)$ times $\mathbf{e}\big((l\lambda+\xi_{l,H},\gamma)\big)$, and the result follows. This proves the lemma.
\end{proof}
Note that the two multipliers in the expression for $\rho_{D}(S)\mathfrak{a}_{\eta,\lambda}^{H}$ in Lemma \ref{rhoDaH} may depend on $\eta \in D$, but Lemma \ref{ortho} shows that their product is well-defined for $\eta \in D/H^{\perp}$. Moreover, the case $H=\{0\}$ of Lemma \ref{rhoDaH} reproduces the presentation from Remark \ref{ecana0}.
\begin{rmk}
Lemma \ref{rhoDaH} suggests that a better indexation for the vector $\mathfrak{a}_{\eta,\lambda}^{H}$ from Equation \eqref{vecH} is using the vector $v:=\binom{\eta}{\lambda} \in D^{2}=\mathbb{Z}^{2}\otimes_{\mathbb{Z}}D$, on which matrices in $\operatorname{SL}_{2}(\mathbb{Z})$ have a natural action. Then the first formula there reads $\rho_{D}(S)\mathfrak{a}_{v}^{H}=\mathbf{e}(-\operatorname{sgn}D/8)\mathbf{e}\big(-(\lambda,\eta)\big)\mathfrak{a}_{Sv}^{H^{\perp}}$, and when the vector $\xi_{l,H}$ from Lemma \ref{xilH} vanishes, the second one becomes $\rho_{D}(T^{l})\mathfrak{a}_{v}^{H}=\mathbf{e}\big(l\tfrac{\lambda^{2}}{2}\big)\mathfrak{a}_{T^{l}v}^{H}$. With a general $\xi_{l,H}$, we can set $T^{l}*v:=T^{l}v+\binom{\xi_{l,H}}{0}$, and get $\rho_{D}(T^{l})\mathfrak{a}_{v}^{H}=\mathbf{e}\big(l\tfrac{\lambda^{2}}{2}\big)\mathfrak{a}_{T^{l}*v}^{H}$. \label{actindvec}
\end{rmk}

The case of isotropic $H$ with $\lambda$ and $\eta$ in $H^{\perp}$ in Lemma \ref{rhoDaH} reproduces the following operator, from \cite{[Br]} and others, in which we denote again the discriminant form $H^{\perp}/H$ by $A$.
\begin{cor}
The map $\uparrow_{H}:\mathbb{C}[A]\to\mathbb{C}[D]$ that is defined by the linear extension of \[\uparrow_{H}\mathfrak{e}_{\mu}:=\tfrac{1}{\sqrt{|H|}}\sum_{\gamma \in H^{\perp},\;\gamma+H=\mu}\mathfrak{e}_{\gamma}\] is an isometric map of representations, embedding $\rho_{A}$ into $\rho_{D}$. \label{arrow}
\end{cor}

\begin{proof}
The vector $\uparrow_{H}\mathfrak{e}_{\mu}$ is, in the notation of Equation \eqref{vecH}, just $\mathfrak{a}_{\eta,\lambda}^{H}$ where $\eta$ and $\lambda$ in $H^{\perp}$ with $\lambda+H=\mu$. These vectors, for $\mu \in A$, are orthonormal vectors that are independent of the representatives (by Lemma \ref{ortho}). The result now follows from the formulae from Lemma \ref{rhoDaH}, together with the fact that $\operatorname{sgn}A=\operatorname{sgn}D$ and the vector $\mathfrak{a}_{-\lambda,\eta}^{H^{\perp}}$ can be viewed, as in Remark \ref{ecana0}, as $\frac{1}{\sqrt{|A|}}\sum_{\tau \in H^{\perp}/H}\mathfrak{a}_{-\lambda,\eta+\tau}^{H}=\frac{1}{\sqrt{|A|}}\sum_{\tau \in H^{\perp}/H}\uparrow_{H}\mathfrak{e}_{\tau}$. This proves the corollary.
\end{proof}

\begin{rmk}
Let $J$ be an isotropic subgroup of $D$, with the associated discriminant form $B:=J^{\perp}/J$, and take $\kappa$ and $\nu$ in $B$. It is then clear from Equation \eqref{vecH} and Corollary \ref{arrow} that for every subgroup $H$ of $D$ that contains $J$ we have $\uparrow_{J}\mathfrak{a}_{\kappa,\nu}^{H/J}=\mathfrak{a}_{\eta,\lambda}^{H}$ for $\eta$ and $\lambda$ in $J^{\perp}$ with respective $B$-image $\kappa$ and $\nu$. In particular this gives the relation $\uparrow_{H}=\uparrow_{J}\circ\uparrow_{H/J}$ of the operators from Corollary \ref{arrow}. \label{compsub}
\end{rmk}

\section{Discriminant Forms with Trivial Quotients \label{TrivQuot}}

A quasi-isotropic subgroup $H$ of $D$ is called \emph{self-dual} if $H^{\perp}=H$. Thus for a self-dual isotropic subgroup $H$, the associated discriminant form $A$ is trivial. Note that ``self-dual'' here does not mean that $H$ is identified with its dual, but rather that when $D$ is the discriminant of a lattice, $H$ corresponds to an over-lattice that is self-dual in the usual sense (\cite{[Bi]} uses the same term).

If $H$ is a self-dual isotropic subgroup of $D$, then the quotient $A=H^{\perp}/H$ is a trivial discriminant form. In this case we have the equality $\operatorname{sgn}D=\operatorname{sgn}A=0$, and $\rho_{D}$ is a representation of $\operatorname{SL}_{2}(\mathbb{Z})$. Then the simple formulae from Lemma \ref{rhoDaH} are enough for establishing a simple formulae for $\rho_{D}(M)$ for every $M\in\operatorname{SL}_{2}(\mathbb{Z})$ using the vectors from Equation \eqref{vecH}.

To do so, write $v:=\binom{\eta}{\lambda}$ for $\eta$ and $\lambda$ in $D$ as in Remark \ref{actindvec}, and for an element $M=\big(\begin{smallmatrix}a & b \\ c & d\end{smallmatrix}\big)\in\operatorname{SL}_{2}(\mathbb{Z})$ and such $v$ we set
\begin{equation}
Q(M,v):=ac\tfrac{\eta^{2}}{2}+bd\tfrac{\lambda^{2}}{2}+bc(\lambda,\eta) \label{QMv}
\end{equation}
(such expressions arise naturally in the general theory of Weil representations, by appropriate substitutions in the general expressions from \cite{[We]}). This expression has the following cocycle property.
\begin{lem}
For $M$ and $N$ in $\operatorname{SL}_{2}(\mathbb{Z})$ and $v \in D^{2}$, the expression from Equation \eqref{QMv} satisfies the equality $Q(MN,v)=Q(N,v)+Q(M,Nv)$ . \label{cocycle}
\end{lem}

\begin{proof}
A tedious calculation can verify this directly, but it suffices to carry out the verification when $M$ is one of the generators $T$ and $S$ of $\operatorname{SL}_{2}(\mathbb{Z})$. If $N=\big(\begin{smallmatrix}a & b \\ c & d\end{smallmatrix}\big)$ and $v=\binom{\eta}{\lambda}$ then substituting the values of $Q(T,Nv)$ and $Q(S,Nv)$ from Equation \eqref{QMv} yields the equalities \[Q(M,v)+\tfrac{(c\eta+d\lambda)^{2}}{2}=(a+c)c\tfrac{\eta^{2}}{2}+(b+d)d\tfrac{\lambda^{2}}{2}+(b+d)c(\lambda,\eta)=Q(TM,v)\] and
\[Q(M,v)-(a\eta+b\lambda,c\eta+d\lambda)=-ac\tfrac{\eta^{2}}{2}-bd\tfrac{\lambda^{2}}{2}-ad(\lambda,\eta)=Q(SM,v),\] as desired. This proves the lemma.
\end{proof}
We can now establish the following result.
\begin{thm}
Assume that $H \subseteq D$ is a self-dual isotropic subgroup. Then for every $M\in\operatorname{SL}_{2}(\mathbb{Z})$ and vector $v \in D^{2}$ the operator $\rho_{D}(M)$ sends the vector $\mathfrak{a}_{v}^{H}$ to $\mathbf{e}\big(Q(M,v)\big)\mathfrak{a}_{Mv}^{H}$. \label{genact}
\end{thm}

\begin{proof}
The fact that the operation on the vectors is a group action combines with Lemma \ref{cocycle} to show that if the formula holds for two matrices $M$ and $N$ and every $v$ then it is valid for the product $MN$. But since $\operatorname{sgn}D=0$, Lemma \ref{rhoDaH} and Remark \ref{actindvec} produce the desired formula for $T$, $T^{-1}$, and $S$, which generate $\operatorname{SL}_{2}(\mathbb{Z})$ multiplicative, and for every vector $v$. This proves the theorem.
\end{proof}

\smallskip

The following well-known lemma allows us to extend Theorem \ref{genact} to the case of a self-dual quasi-isotropic subgroup $H$, which is not necessarily isotropic.
\begin{lem}
The elements $T^{2}$ and $S$ of the double cover $\operatorname{Mp}_{2}(\mathbb{Z})$ of $\operatorname{SL}_{2}(\mathbb{Z})$ generate a subgroup $\Gamma_{\mathrm{odd}}$ of index 3, which is the semi-direct product in which $\langle S \rangle$ acts by conjugation on the free group generated by $T^{2}$ and $ST^{2}S^{-1}$. The non-trivial cosets of $\Gamma_{\mathrm{odd}}$ in $\operatorname{Mp}_{2}(\mathbb{Z})$ are represented by $T$ and $ST$. \label{Gammaodd}
\end{lem}
Consider now a character $\chi$ of $\Gamma_{\mathrm{odd}}$ that is trivial on the free subgroup from Lemma \ref{Gammaodd}, and is thus determined by the 8th root of unity $\chi(S)$. We extend $\chi$ to $\operatorname{Mp}_{2}(\mathbb{Z})$ by taking elements $TM$ and $STM$ with $M\in\Gamma_{\mathrm{odd}}$ to $\chi(M)$ and $\chi(S)\chi(M)$ respectively, and we define a twisted operation of $\operatorname{Mp}_{2}(\mathbb{Z})$ on $D^{2}$ by
\begin{equation}
M*v:=\begin{cases}Mv, & \mathrm{when\ }M\in\Gamma_{\mathrm{odd}}, \\ Mv+\binom{\xi_{H}}{0}, & \mathrm{in\ case\ }M \in T\Gamma_{\mathrm{odd}} \\ Mv+\binom{0}{\xi_{H}}, & \mathrm{if\ }M \in ST\Gamma_{\mathrm{odd}},\end{cases} \label{twistact}
\end{equation}
where $\xi_{H}$ is the vector from Lemma \ref{xiHdef}. We also modify the cocycle from Lemma \ref{cocycle}, and define, for $M=\big(\begin{smallmatrix}a & b \\ c & d\end{smallmatrix}\big)\in\operatorname{SL}_{2}(\mathbb{Z})$ and $v$ as above, the expression
\begin{equation}
\tilde{Q}_{H}(M,v):=\begin{cases}Q(M,v), & \mathrm{if\ }M\in\Gamma_{\mathrm{odd}} \cup T\Gamma_{\mathrm{odd}} \\ Q(M,v)+(\xi_{H},a\eta+b\lambda), & \mathrm{when\ }M \in ST\Gamma_{\mathrm{odd}}.\end{cases} \label{twistQ}
\end{equation}
Using these expressions we obtain the following generalization of Theorem \ref{genact}.
\begin{thm}
Let $H$ be a self-dual quasi-isotropic subgroup of $D$, and consider the map $\chi:\operatorname{Mp}_{2}(\mathbb{Z})\to\mathbb{C}^{\times}$ extending the character of $\Gamma_{\mathrm{odd}}$ that sends $S$ to $\mathbf{e}(-\operatorname{sgn}D/8)$. Then we have the equality $\rho_{D}(M)\mathfrak{a}_{v}^{H}=\chi(M)\mathbf{e}\big(\tilde{Q}_{H}(M,v)\big)\mathfrak{a}_{M*v}^{H}$ for every $M\in\operatorname{Mp}_{2}(\mathbb{Z})$ and $v \in D^{2}$, with $M*v$ and $\tilde{Q}_{H}(M,v)$ from Equations \eqref{twistact} and \eqref{twistQ} respectively. \label{qisotriv}
\end{thm}

\begin{rmk}
Note that while the expression $M*v$ for $M\notin\Gamma_{\mathrm{odd}}$ and $\tilde{Q}_{H}(M,v)$ when $M \in ST\Gamma_{\mathrm{odd}}$ depend on the choice of $\xi_{H}$, Lemma \ref{ortho} implies that $\mathfrak{a}_{M*v}^{H}$ with $M \in T\Gamma_{\mathrm{odd}}$ and the combination $\mathbf{e}\big(\tilde{Q}_{H}(M,v)\big)\mathfrak{a}_{M*v}^{H}$ for $M \in ST\Gamma_{\mathrm{odd}}$ only depend on the image of $\xi_{H}$ in $D/H$, which is canonical by Lemma \ref{xiHdef}. Moreover, if $H$ is isotropic then $\chi$ is trivial and Equations \eqref{twistact} and \eqref{twistQ} reduce to $M*v=Mv$ and Equation \eqref{QMv} for all $M\in\operatorname{Mp}_{2}(\mathbb{Z})$, so that Theorem \ref{qisotriv} reproduces Theorem \ref{genact} in this case. See Remark \ref{discqiso} below for the other cases. \label{actwdqiso}
\end{rmk}

\begin{proof}
Lemma \ref{xilH} implies that $\xi_{l,H}$ vanishes for even $l$ and equals $\xi_{H}$ when $l$ is odd. Thus Lemma \ref{rhoDaH} and Remark \ref{actindvec} extend the proof of Theorem \ref{genact} to the case of $M\in\Gamma_{\mathrm{odd}}=\langle S,T^{2} \rangle$, up to scalar multiples coming from the fact that $\rho_{D}(S)$ has the additional multiplier $\mathbf{e}(-\operatorname{sgn}D/8)=\chi(S)$. As $\rho(S^{-1})$ comes with the inverse multiplier, we deduce that the formula for $M=ST^{2}S^{-1}$ involves no additional factors, so that the formula from Theorem \ref{genact} is valid for elements of the free subgroup from Lemma \ref{Gammaodd}. It is now clear that the asserted formula holds for every $M\in\Gamma_{\mathrm{odd}}$ (with our character $\chi$), and note that Remark \ref{actindvec} with $l=1$ (and $\xi_{1,H}=\xi_{H}$) and Equation \eqref{twistact} give we have $(TN)*v=T*Nv$ and $(STN)*v=S\big((TN)*v\big)$ for $N\in\Gamma_{\mathrm{odd}}$. This establishes the desired result when $M=TN \in T\Gamma_{\mathrm{odd}}$, and since the proof of Lemma \ref{cocycle} combines with Equations \eqref{twistact} and \eqref{twistQ} to show that $Q(TN,v)+Q(S,TN*v)=\tilde{Q}_{H}(STN,v)$ for such $N$, the formula for $M=STN \in ST\Gamma_{\mathrm{odd}}$ as well. This completes the proof of the theorem.
\end{proof}

\begin{rmk}
When $H$ is quasi-isotropic and self-dual, Corollary \ref{isoinqiso} yields the subgroup $H_{0}$, and the associated discriminant form $A_{0}:=H_{0}^{\perp}/H_{0}$. The latter is trivial when $H$ is isotropic, but otherwise it has order 4. It is therefore either cyclic, where $\chi(S)=\mathbf{e}(-\operatorname{sgn}A_{0})$ is of order 8 and $H/H_{0}$ is the unique subgroup of order 2, or isomorphic to the Klein 4-group. In the latter case we need the element $\tau \in A_{0}$ generatic $H/H_{0}$ to satisfy $\frac{\tau^{2}}{2}=\frac{1}{2}+\mathbb{Z}$, so that using the notation from \cite{[Sch]}, \cite{[Str]}, \cite{[Ze]}, and others, $A_{0}$ is either isomorphic to $2_{\pm2}^{+2}$ with a unique choice of $\tau$ and $\chi(S)$ of order 4, or to $2^{+2}_{II}$ with a unique $\tau$ and $\chi(S)=1$, or to $2^{-2}_{II}$, where $\chi(S)=-1$ and $\tau$ can be any non-trivial element. \label{discqiso}
\end{rmk}

\smallskip

We conclude this section by modifying the vectors from Equation \eqref{vecH} in order to respect the actions of automorphisms. Let $D$ be any discriminant form, and consider a subgroup $H$ of $D$, a group $G$ of automorphisms of $D$ (that preserve the quadratic form) that fixes $H$ (and thus also $H^{\perp}$), and a character $\psi:G\to\mathbb{C}^{\times}$. For a pair of vectors $\eta$ and $\lambda$ in $D$, denote by $G_{\eta,\lambda}^{H}$ the subgroup of $G$ that stabilizes the cosets $\eta+H^{\perp} \in D/H^{\perp}$ and $\lambda+H \in D/H$. Then we define
\begin{equation}
\mathfrak{a}_{\eta,\lambda}^{H,\psi}=\frac{1}{\sqrt{|G|\cdot|G_{\eta,\lambda}^{H}|}}\sum_{\varphi \in G}\psi^{-1}(\varphi)\mathfrak{a}_{\varphi(\eta),\varphi(\lambda)}^{H}. \label{withaut}
\end{equation}
Note that for $\phi \in G_{\eta,\lambda}^{H}$, Lemma \ref{ortho} presents $\mathfrak{a}_{\varphi\phi(\eta),\varphi\phi(\lambda)}^{H}$ as $\mathfrak{a}_{\varphi(\eta),\varphi(\lambda)}^{H}$ times a factor that depends only on $\phi$, so that the vector from Equation \eqref{withaut} vanishes unless this multiple is $\psi(\phi)$ (and then it can be presented as $\sqrt{|G_{\eta,\lambda}^{H}|/|G|}$ times a sum over $G/G_{\eta,\lambda}^{H}$).

A special case of interest is where $G=\{\pm\operatorname{Id}_{D}\}$, which fixes every $H$, and the character $\psi$ is simply a sign $\pm1$. This case is related to the action of the central element $Z$ of $\operatorname{Mp}_{2}(\mathbb{Z})$ decomposing the representation space $\mathbb{C}[D]$ into the symmetric and anti-symmetric elements. In the case where $2\eta \not\in H^{\perp}$ or $2\lambda \notin H$, i.e., when $G_{\eta,\lambda}^{H}$ is trivial, Equation \eqref{withaut} simplifies to
\begin{equation}
\mathfrak{a}_{\eta,\lambda}^{H,\pm}=\frac{\mathfrak{a}_{\eta,\lambda}^{H}\pm\mathfrak{a}_{-\eta,-\lambda}^{H}}{\sqrt{2}}=\begin{cases} \displaystyle{\sum_{\gamma \in H}\tfrac{\mathbf{e}((\gamma,\eta))}{\sqrt{2|H|}}(\mathfrak{e}_{\lambda+\gamma}\pm\mathfrak{e}_{-\lambda-\gamma})}, & \mathrm{when\ }2\lambda \notin H, \\ \displaystyle{\sum_{\gamma \in H}\tfrac{\mathbf{e}((\gamma,\eta))\pm\mathbf{e}(-(2\lambda+\gamma,\eta))}{\sqrt{2|H|}}\mathfrak{e}_{\lambda+\gamma}}, & \mathrm{if\ }2\lambda \in H,\ 2\eta \not\in H^{\perp} \end{cases} \label{symvec}
\end{equation}
(the second expression comes from Lemma \ref{ortho}). When $2\lambda \in H$ and $2\eta \in H^{\perp}$, i.e., the case where $G_{\eta,\lambda}^{H}=G$, the number $\mathbf{e}\big((2\lambda,\eta)\big)$ is a sign, and Equation \eqref{withaut} collapses $\mathfrak{a}_{\eta,\lambda}^{H,\pm}$ to simply $\mathfrak{a}_{\eta,\lambda}^{H}$ if $\mathbf{e}\big((2\lambda,\eta)\big)=\pm1$ and to 0 when it is the opposite sign.

Many results from before extend to these more general vectors.
\begin{prop}
Fix $G$ and $\psi$. Then Lemma \ref{ortho} is valid for the vectors from Equation \eqref{withaut}, with the condition for the non-orthogonality of $\mathfrak{a}_{\eta,\lambda}^{H,\psi}$ and $\mathfrak{a}_{\kappa,\nu}^{H,\psi}$ being that the action of $G$ does not relate $(\eta+H^{\perp},\lambda+H)$ to $(\kappa+H^{\perp},\nu+H)$. The formulae from Lemma \ref{rhoDaH} also remain valid when the superscript $\varepsilon$ added throughout. Moreover, if $H$ is a self-dual isotropic subgroup, then Theorems \ref{genact} and \ref{qisotriv} hold for the vectors with $\varepsilon$ as well. \label{propsym}
\end{prop}

\begin{proof}
The first and second assertions follow from the pairings and quadratic form values being invariant under $G$, as well as the fact that $2\xi_{l,H} \in H^{\perp}$ and $l\lambda \in H^{\perp}$ for $\lambda \in H$ in the situation from Lemma \ref{xilH}. This also implies that the expression from Equation \eqref{QMv} is invariant under $G$, and since the action of $G$ commutes with that of $\operatorname{Mp}_{2}(\mathbb{Z})$, this establishes the last assertion in the isotropic case, as well as in the quasi-isotropic one for $M\in\Gamma_{\mathrm{odd}}$. For $M \in T\Gamma_{\mathrm{odd}}$ the same argument combines with the fact that $2\xi_{H} \in H^{\perp}=H$ to give the desired result, and when $M \in ST\Gamma_{\mathrm{odd}}$, Remark \ref{actwdqiso} implies that replacing $\xi_{H}$ by $\varphi(\xi_{H})$ (which lies in the same coset modulo $H$ by Lemma \ref{xiHdef}) in the summand associated with $\varphi$ does not change the result. Thus the assertion is true also in the remaining cases. This proves the proposition.
\end{proof}

Such characters can be used to determine the complete decomposition of some Weil representations into irreducible components. For example, assume that $D$ is a \emph{cyclic} discriminant form, of size $N$, and fix a generator $\gamma$ of $D$. Then $\frac{\gamma^{2}}{2}$ equals $\frac{t}{N}+\mathbb{Z}$ when $N$ is odd and $\frac{t}{2N}+\mathbb{Z}$ in case $N$ is even, where $t$ is prime to $N$. The automorphism group $G$ of $D$ is a product of $\{\pm1\}$'s, one for each prime dividing $N$, except for $p=2$ when $N$ is even but indivisible by 4. The subgroups of $D$ are determined by their cardinality, which is a divisor $M$ of $N$, and the subgroup $H_{M}$ of cardinality $M$ is isotropic if and only if $M^{2}$ divides $N$ and $\frac{N}{M^{2}}$, which is the size of the quotient $A_{M}:=H_{M}^{\perp}/H_{M}$, has the same parity of $N$. We say that a character $\psi$ of $G$ is \emph{admissible for $M$} if it attains $+1$ on the $-1$ component associated with any prime $p$ that does not divide $\frac{N}{M^{2}}$, as well as with $p=2$ in case $4$ divides $N$ but does not divide $\frac{N}{M^{2}}$. For such a discriminant form we obtain the following decomposition.
\begin{thm}
The set of irreducible representations of $\rho_{D}$ is in one-to-one correspondence with pairs $(M,\psi)$ where $M$ is such a divisor of $N$ and $\psi$ is a character of $G$ that is admissible for $M$. The sub-representation associated with such a pair $M$ and $\psi$ consists of those elements of $\mathbb{C}[D]$ on which $G$ operates via $\psi$, which are in the image of $\uparrow_{H_{M}}\mathbb{C}[A_{M}]$, and which are perpendicular to $\uparrow_{H_{L}}\mathbb{C}[A_{L}]$ for any divisor $L$ of $N$ which is properly divisible by $M$. \label{cycdecom}
\end{thm}
Indeed, we have a surjective map from $G$ onto the automorphism group of $A_{M}$ for each $M$, and the kernel of this map consists precisely of those $\{\pm1\}$'s that are associated with the primes in the definition of admissibility. Hence Remark \ref{compsub} restricts the proof of Theorem \ref{cycdecom} to verifying the irreducibility of the representations associated with $M=1$, where it is clear that vectors on which $G$ operates via $\psi$ exist for every $\psi$ (check the vectors obtained from a generator). Then one can verify that the space in question admits a basis consisting of eigenvectors for $\rho_{D}(T)$ with different eigenvalues, and applying $\rho_{D}(S)$ to each one of them gives a linear combination involving all the different eigenvalues, and the irreducibility easily follows. It is clear that neither the description of the subgroups of $D$, nor the result of Theorem \ref{cycdecom}, hold when $D$ is not cyclic.

\section{Quotients of Prime Exponent \label{PrimQuot}}

Some discriminant forms $D$ do not contain self-dual quasi-isotropic subgroups, and for a subgroup $H$ that is not quasi-isotropic and self-dual, Lemma \ref{rhoDaH} indicates that bases in which the action of $\operatorname{Mp}_{2}(\mathbb{Z})$ look particularly simple may need to involve more that one group (e.g., $H$ and $H^{\perp}$). For doing so we shall need some additional formulae.
\begin{lem}
Assume that $H$ is contained in another subgroup $K$ of $D$, and let $\mathfrak{R}$ be a set of representatives for $K/H$ in $K$. Then for $\eta$ and $\lambda$ in $D$ we have \[\mathfrak{a}_{\eta,\lambda}^{K}=\frac{1}{\sqrt{|K/H|}}\sum_{\tau\in\mathfrak{R}}\mathbf{e}\big((\tau,\eta)\big)\mathfrak{a}_{\eta,\lambda+\tau}^{H}.\] If, in addition, we take $l$ and $\xi_{l,H}$ as in Lemma \ref{xilH}, then we have the equality \[\rho_{D}(T^{l})\mathfrak{a}_{\eta,\lambda}^{K}=\frac{1}{\sqrt{|K/H|}}\mathbf{e}\big(l\tfrac{\lambda^{2}}{2}\big)\sum_{\tau\in\mathfrak{R}}\mathbf{e}\big((\tau,\eta+l\lambda)+l\tfrac{\tau^{2}}{2}\big) \mathfrak{a}_{\eta+l\lambda+l\tau+\xi_{l,H},\lambda+\tau}^{H}.\] The summands in these formulae are independent of the choice of $\mathfrak{R}$, and in case $K \subseteq H^{\perp}$ we can omit $l\tau$ from the first index in the latter expression. \label{actTKH}
\end{lem}

\begin{proof}
The first formula follows from re-ordering the sum in Equation \eqref{vecH}, the second one follows easily from the first via Lemma \ref{rhoDaH}, and the remaining ones are now simple consequences of Lemma \ref{ortho} and the definition of $\xi_{l,H}$ in Lemma \ref{xilH}. This proves the lemma.
\end{proof}

We shall need these formulae when $H$ is quasi-isotropic and $K=H^{\perp}$, for powers $l$ satisfying some co-primality conditions. But first we shall need the following extension of Milgram's formula. Let $H \subseteq D$ be quasi-isotropic, with the isotropic subgroup $H_{0}$ from Corollary \ref{isoinqiso} and the associated discriminant form $A_{0}$. Take an integer $l$ that is prime to $|H^{\perp}/H|$, and then if $H$ is isotropic or $l$ is odd then we write $A_{0}(l)$ for the re-scaling of $A_{0}$ by $l$. For $H$ not isotropic and even $l$, multiplying the quadratic form on $H^{\perp}$ by $l$ transforms $A$ to a discriminant form, and we allow the abuse of notation of writing $A_{0}(l)$ for the resulting discriminant form there as well.

Let $\xi_{l,H}$ be the vector from Lemma \ref{xilH}, and choose $k\in\mathbb{Z}$ with the following properties. In case $l$ is odd or $H$ is isotropic, we require that $kl$ is congruent to 1 modulo the denominators of $(\gamma,\xi_{l,H})$ and $\frac{\gamma^{2}}{2}$ for $\gamma \in H^{\perp}$, as well as that of $\frac{\xi_{l,H}^{2}}{2}$. When $H$ is not isotropic and $l$ is even, so that $\xi_{l,H} \in H^{\perp}$ and $|H^{\perp}/H|$ is odd, we impose these congruences only modulo the odd parts of these denominators, and demand that $k$ be even.
\begin{lem}
Given such $H$ and $l$, take $\xi_{l,H}$ and $k$ as defined above. Then each term in the sum $\sum_{\sigma \in H^{\perp}/H}\mathbf{e}\big(l\frac{\sigma^{2}}{2}-(\sigma,\xi_{l,H})\big)$ is well-defined, and the value of the entire sum is $\mathbf{e}\big(-k\frac{\xi_{l,H}^{2}}{2}\big)\mathbf{e}\big(\operatorname{sgn}A_{0}(l)/8\big)\cdot\sqrt{|H^{\perp}/H|}$. \label{Milodd}
\end{lem}

\begin{proof}
The invariance under changing $\sigma \in H^{\perp}$ by an element of $H$ follows directly from Lemma \ref{xilH}, making each summand well-defined.

When $H$ is isotropic or $l$ is even we have $\xi_{l,H} \in H^{\perp}$, and then writing the variable $\sigma$ as $\gamma+k\xi_{l,H}$ with $\gamma \in A:=H^{\perp}/H$ gives $\mathbf{e}\big(-k\frac{\xi_{l,H}^{2}}{2}\big)$ times the left hand side of Equation \eqref{Milgram}, with $D$ replaced by $A_{0}(l)$ (this is $A(l)$ if $H$ is isotropic). The result thus follows from Milgram's formula in this case.

If $H$ is quasi-isotropic and $l$ is odd, then we have $\frac{\gamma^{2}}{2}=\frac{1}{2}$ for any $\gamma \in H \setminus H_{0}$, so by fixing such an element we have $H=H_{0}\cup(\gamma+H_{0})$. Since Corollary \ref{isoinqiso} gives $H_{0}^{\perp}=H^{\perp}\cup(\xi_{H}+H^{\perp})$, $k$ must be odd as well, the non-trivial coset can be written as $k\xi_{H}+H^{\perp}$. The left hand side of Milgrams's formula for $A_{0}(l)$ can thus be written as the sum of $\sum_{\sigma \in H^{\perp}/H_{0}}\mathbf{e}\big(l\frac{\sigma^{2}}{2}\big)$ and $\sum_{\sigma \in H^{\perp}/H_{0}}\mathbf{e}\big(l\frac{(\sigma+k\xi_{H})^{2}}{2}\big)$. Now, replacing $\sigma$ by $\sigma+\gamma$ in the former sum inverts all the summands, which implies that this sum vanishes. Thus Equation \eqref{Milgram} compares $\mathbf{e}(\operatorname{sgn}A_{0}(l))\cdot\sqrt{|H_{0}^{\perp}/H_{0}|}$ with $\mathbf{e}\big(k\frac{\xi_{H}^{2}}{2}\big)$ times $\sum_{\sigma \in H^{\perp}/H_{0}}\mathbf{e}\big(l\frac{\sigma^{2}}{2}-(\sigma,\xi_{H})\big)$. But the well-definedness of the required sum means that it equals half of the latter sum, which gives the desired result by the index 2 property in Corollary \ref{isoinqiso} and the fact that $\xi_{l,H}=\xi_{H}$ in this setting. This proves the lemma.
\end{proof}

In this case the formulae from Lemmas \ref{rhoDaH} and \ref{actTKH} are complemented by the following evaluation.

\begin{prop}
Let $H$, $l$, $A_{0}(l)$, $\xi_{l,H}$, and $k$ be as in Lemma \ref{Milodd}, and take $\eta$ and $\lambda$ in $D$. Then $\rho_{D}(ST^{l})\mathfrak{a}_{\eta,\lambda}^{H^{\perp}}$ equals $\mathbf{e}\big(\frac{\operatorname{sgn}A_{0}(l)-\operatorname{sgn}D}{8}\big)$ times
\[\mathbf{e}\big[(kl-1)\big(l\tfrac{\lambda^{2}}{2}+(\lambda,\eta+\xi_{l,H})\big)+k\tfrac{\eta^{2}}{2}+k(\eta,\xi_{l,H})\big]\rho_{D}(T^{-k})\mathfrak{a}_{(kl-1)\lambda+k\eta,\eta+l\lambda+\xi_{l,H}}^{H^{\perp}}.\] \label{actofS}
\end{prop}

\begin{proof}
Taking the second formula from Lemma \ref{actTKH} with $K=H^{\perp}$ and applying $\rho_{D}(S)$ presents, via Lemma \ref{rhoDaH}, the vector $\rho_{D}(ST^{l})\mathfrak{a}_{\eta,\lambda}^{H^{\perp}}$ as \[\frac{\mathbf{e}(-\operatorname{sgn}D/8)}{\sqrt{|H^{\perp}/H|}}\mathbf{e}\big(-l\tfrac{\lambda^{2}}{2}-(\lambda,\eta+\xi_{l,H})\big)\sum_{\tau\in\mathfrak{R}}\mathbf{e}\big(-(\tau,\xi_{l,H})+l\tfrac{\tau^{2}}{2}\big) \mathfrak{a}_{-\lambda-\tau,\eta+l\lambda+\xi_{l,H}}^{H^{\perp}}.\] Using the first formula from Lemma \ref{actTKH}, the sum over $\tau$ becomes \[\frac{1}{\sqrt{|H^{\perp}/H|}}\sum_{\rho\in\mathfrak{R}}\mathbf{e}\big(-(\lambda,\rho)\big)\Bigg[\sum_{\tau\in\mathfrak{R}}\mathbf{e}\big(l\tfrac{\tau^{2}}{2}-(\tau,\rho+\xi_{l,H})\big)\Bigg] \mathfrak{a}_{-\lambda,\eta+l\lambda+\xi_{l,H}+\rho}^{H},\] where Lemma \ref{ortho} allowed us to ignore $\tau \in H^{\perp}$ in the first index before interchanging the summation order.

Now, we may replace $\mathfrak{R}$ by $H^{\perp}/H$ in the internal sum, and writing $\tau$ as $\sigma+k\rho$ with $\sigma \in H^{\perp}/H$ transforms the latter sum into $\mathbf{e}\big(-k\tfrac{\rho^{2}}{2}-k(\rho,\xi_{l,H})\big)$ times the expression from Lemma \ref{Milodd}. The considerations from the proof of that lemma thus present $\rho_{D}(ST^{l})\mathfrak{a}_{\eta,\lambda}^{H^{\perp}}$ as $\mathbf{e}\big(\frac{\operatorname{sgn}A_{0}(l)-\operatorname{sgn}D}{8}\big)\big/\sqrt{|H^{\perp}/H|}$ times \[\mathbf{e}\big(-l\tfrac{\lambda^{2}}{2}-(\lambda,\eta+\xi_{l,H})-k\tfrac{\xi_{l,H}^{2}}{2}\big)\sum_{\rho\in\mathfrak{R}}\mathbf{e}\big(-(\rho,\lambda+k\xi_{l,H})-k\tfrac{\rho^{2}}{2}\big) \mathfrak{a}_{-\lambda,\eta+l\lambda+\xi_{l,H}+\rho}^{H}.\] Lemma \ref{ortho} allows us to add $(1-k)\xi_{l,H} \in H^{\perp}$ to $-\lambda$ (note that $(1-k)\xi_{l,H} \in H^{\perp}$ also when $\xi_{l,H} \notin H^{\perp}$ since then $k$ is odd), and since the resulting expression equals $\beta-k\alpha+\xi_{l,H}$ for $\beta:=(kl-1)\lambda+k\eta$ and $\alpha:=\eta+l\lambda+\xi_{l,H}$, and $-\lambda-k\xi_{l,H}$ equals $\beta-k\alpha$, Lemma \ref{actTKH} shows that the sum over $\mathfrak{R}$ here equals $\sqrt{|H^{\perp}/H|}\mathbf{e}\big(k\tfrac{\alpha^{2}}{2}\big)$ times $\rho_{D}(T^{-k})\mathfrak{a}_{\beta,\alpha}^{H^{\perp}}$. Substituting the value of $\alpha$ and simple algebra now produces the desired result.
\end{proof}

\begin{rmk}
Proposition \ref{propsym} can be extended, with a similar proof, to show that the formulae from Proposition \ref{actofS} are also valid for the vectors from Equations \eqref{withaut} and \eqref{symvec}. \label{rhoSsym}
\end{rmk}

\smallskip

We shall be using these expressions in the case where $H$ is a quasi-isotropic subgroup of $D$ such that the quotient $H^{\perp}/H$ has prime exponent $p$. We shall need the following technical lemma.
\begin{lem}
Let $H_{0}$ be the subgroup from Lemma \ref{xiHdef}. If $H^{\perp}/H$ has prime exponent $p$ then there exists a subgroup $\tilde{H}^{\perp}$ of $H^{\perp}$ such that $\tilde{H}^{\perp}/H_{0}$ is a complement of $H/H_{0}$ inside the finer quotient $H^{\perp}/H_{0}$. \label{liftmodH0}
\end{lem}

\begin{proof}
If $p$ is odd then the fact that $H/H_{0}$ has order dividing 2 determines $\tilde{H}^{\perp}/H_{0}$ as the subgroup containing those elements of $H^{\perp}/H_{0}$ whose order is odd. For $p=2$ we note that $2(\gamma,\delta)$ vanishes for every $\gamma$ and $\delta$ in $H^{\perp}$ (since $2\delta \in H$), implying that $H^{\perp}$ and 2 satisfy the condition from Lemma \ref{xilH}. By this means that $\frac{\gamma^{2}}{2}\in\frac{1}{4}\mathbb{Z}\big/\mathbb{Z}$ for every $\gamma \in H^{\perp}$, and therefore $2\gamma$ lies in the subgroup $H_{0}$ from Lemma \ref{xiHdef}. Hence $H^{\perp}/H_{0}$ also has exponent 2, and every subgroup in it has a complement. This proves the lemma.
\end{proof}

\begin{cor}
There exists a vector $\xi_{H,\tilde{H}} \in H_{0}^{\perp}$ which pairs trivially with $\tilde{H}^{\perp}/H_{0}$, but not trivially with $H/H_{0}$ in case the latter group is non-trivial. This vector is uniquely determined modulo $H$, has order at most 2 in $D/H$, and the order of $\frac{\xi_{H,\tilde{H}}^{2}}{2}$ in $\mathbb{Q}/\mathbb{Z}$ divides 8. \label{veccomp}
\end{cor}

\begin{proof}
The first statement follows directly from Lemma \ref{liftmodH0}, and the second one as in Lemma \ref{xilH}. Next we note that $2\xi_{H,\tilde{H}}$ is perpendicular to $H^{\perp}$ hence lies in $H$, and the remaining assertions follows. This proves the corollary.
\end{proof}
It follows immediately from Corollary \ref{veccomp} that the vector $\xi_{H,\tilde{H}}$ can serve as a representative for $\xi_{H}$ from Lemma \ref{xiHdef}. Moreover, given $l\in\mathbb{Z}$, we recall from that Lemma that $\xi_{l,H}$ vanishes if $l$ is even and equals $\xi_{H}$ for odd $l$.  Using the third assertion in that corollary, we therefore set $\xi_{l,H,\tilde{H}}$ to be 0 in case $l$ is even and $\xi_{H,\tilde{H}}$ when $l$ is odd. We remark that if $p$ is odd then $\xi_{l,H,\tilde{H}}$ is simply $\xi_{pl,H^{\perp}}$ from Lemma \ref{xiHdef}, but this is not true when $p=2$.

Assume thus that $H$ is such a subgroup, fix $\tilde{H}^{\perp}$ as in Lemma \ref{liftmodH0}, and assume that $J$ is subgroup of $H$, of order $p$. Then the $\mathbb{Q}/\mathbb{Z}$-dual $D/J^{\perp}$ of $J$ is also cyclic of order $p$, and the image of any element of $D \setminus J^{\perp}$ generates it. It follows that given two elements $\eta$ and $\lambda$ in $D$, if $\lambda \notin J^{\perp}$, then there exists $l\in\mathbb{Z}$ such that $\eta-l\lambda \in J^{\perp}$. The subgroup $\tilde{H}^{\perp}$ thus produces the vector $\xi_{l,H,\tilde{H}}$. Denote $A=H^{\perp}/H$ (also when $H$ is not isotropic), and we set
\begin{equation}
\mathfrak{b}_{\eta,\lambda}^{H,J}:=\begin{cases} \frac{1}{\sqrt{|A|}}\sum_{\tau \in A}\mathbf{e}\big((\tau,\eta-\xi_{l,H,\tilde{H}})+l\frac{\tau^{2}}{2}\big)\mathfrak{a}_{\eta,\lambda+\tau}^{H}, & \mathrm{if\ }\lambda \notin J^{\perp}, \\ \mathfrak{a}_{\eta,\lambda}^{H}, & \mathrm{for\ }\lambda \in J^{\perp},\end{cases} \label{bHJvec}
\end{equation}
where we could take $\tau$ in $A$ rather than in a representing set by Lemmas \ref{ortho} and \ref{xilH}. It is clear that for fixed $l$ the expression from Equation \eqref{bHJvec} depends only on the class of $\xi_{l,H,\tilde{H}}$ in $D/H$, and for showing that it is well-defined as an formula of $\eta$ and $\lambda$ alone, we prove the following lemma.
\begin{lem}
Given two elements $\eta \in D$ and $\lambda \in D \setminus J^{\perp}$, the vector $\mathfrak{b}_{\eta,\lambda}^{H,J}$ equals $\mathbf{e}\big(-l\tfrac{\lambda^{2}}{2}\big)\rho_{D}(T^{l})\mathfrak{a}_{\eta-l\lambda-\xi_{l,H,\tilde{H}},\lambda}^{H^{\perp}}$, and thus depends on $l$ only modulo $p$. \label{bwelldef}
\end{lem}

\begin{proof}
The expression for $\mathfrak{b}_{\eta,\lambda}^{H,J}$ is a consequence of Lemma \ref{actTKH} and Equation \eqref{bHJvec}. The invariance under changing $l$ by a multiple of $p$ follows from Remark \ref{actindvec} and the fact that the difference between $\xi_{l,H,\tilde{H}}$ and $\xi_{l+np,H,\tilde{H}}$ is just $\xi_{np,H,\tilde{H}}$ from Corollary \ref{veccomp}. This proves the lemma.
\end{proof}
From Lemma \ref{bwelldef} we also obtain the following analogue of Lemma \ref{ortho}.
\begin{cor}
For $\lambda \in J^{\perp}$, the vector $\mathfrak{b}_{\eta,\lambda}^{H,J}$ from Equation \eqref{bHJvec} depends only on the image of $\eta$ modulo $H^{\perp}$, while adding $\delta \in H \subseteq J^{\perp}$ to $\lambda$ multiplies it by $\mathbf{e}\big(-(\delta,\eta)\big)$. On the other hand, if $\lambda \notin J^{\perp}$ then this vector depends on the image of $\eta$ in $D/H$, and for $\delta \in H^{\perp}$ we have $\mathfrak{b}_{\eta,\lambda+\delta}^{H,J}=\mathbf{e}\big(l\frac{\delta^{2}}{2}-(\delta,\eta-\xi_{l,H,\tilde{H}})\big)\mathfrak{b}_{\eta-l\delta,\lambda}^{H,J}$ with $l$ as in Equation \eqref{bHJvec}. Two vectors $\mathfrak{b}_{\eta,\lambda}^{H,J}$ and $\mathfrak{b}_{\kappa,\nu}^{H,J}$ with indices that are not related by these transformations are orthogonal, and each $\mathfrak{b}_{\eta,\lambda}^{H,J}$ has norm 1. \label{basisTl}
\end{cor}

The actions of the generators $S$ and $T$ on the vectors from Equation \eqref{bHJvec} take the following form.
\begin{prop}
For every such $H$, $J$, $\eta$, and $\lambda$ we have the equality \[\rho_{D}(T)\mathfrak{b}_{\eta,\lambda}^{H,J}=\mathbf{e}\big(\tfrac{\lambda^{2}}{2}\big)\mathfrak{b}_{\eta+\lambda+\xi_{H,\tilde{H}},\lambda}^{H,J}.\] If either $\eta$ or $\lambda$ lies outside of $J^{\perp}$, then there exists an 8th root of unity $\varepsilon_{J}(S,v)$, for $v:=\binom{\eta}{\lambda}$ as in Remark \ref{actindvec}, such that \[\rho_{D}(S)\mathfrak{b}_{\eta,\lambda}^{H,J}=\varepsilon_{J}(S,v)\mathbf{e}\big(-(\lambda,\eta)\big)\mathfrak{b}_{-\lambda,\eta}^{H,J}.\] \label{formbHJ}
\end{prop}

\begin{proof}
If $\lambda \in J^{\perp}$ then both formulae follow from Lemma \ref{rhoDaH} (note that for $S$ we assume $\eta \notin J^{\perp}$, and we can apply the formula from Lemma \ref{bwelldef}), the second one with $\varepsilon_{J}(S,v)=\mathbf{e}(-\operatorname{sgn}D/8)$. When $\lambda \notin J^{\perp}$, the expression for $T$ is a consequence of Lemma \ref{bwelldef}, since the difference between $\xi_{l,H,\tilde{H}}$ and $\xi_{l+1,H,\tilde{H}}$ is $\xi_{H,\tilde{H}}$, which has order 2 in $D/H$ by Corollary \ref{veccomp}.

For evaluating the action of $\rho_{D}(S)$ when $\lambda \notin J^{\perp}$, we take $l$ as in Equation \eqref{bHJvec}, and write $\mathfrak{b}_{\eta,\lambda}^{H,J}$ as in Lemma \ref{bwelldef} once again. If $\eta \in J^{\perp}$ then $l$ can be taken to be 0, and the desired formula is again obtained from Lemma \ref{rhoDaH}, again with $\varepsilon_{J}(S,v)=\mathbf{e}(-\operatorname{sgn}D/8)$. On the other hand, when $\eta \notin J^{\perp}$ the index $l$ is prime to the $p$-power $|A|$, we choose $k$ as in Lemma \ref{Milodd}, and then Proposition \ref{actofS} presents $\rho_{D}(S)\mathfrak{b}_{\eta,\lambda}^{H,J}$, after some cancelations, as \[\mathbf{e}\Big(\tfrac{\operatorname{sgn}A_{0}(l)-\operatorname{sgn}D}{8}-(\lambda,\eta)-k\tfrac{\xi_{l,H,\tilde{H}}^{2}}{2}+k\tfrac{\eta^{2}}{2}\Big)\rho_{D}(T^{-k}) \mathfrak{a}_{-\lambda+k\eta-k\xi_{l,H,\tilde{H}},\eta}^{H^{\perp}}.\] Another application of Lemma \ref{bwelldef} compares this with the desired result, with $\varepsilon_{J}(S,v)=\mathbf{e}\Big(\frac{\operatorname{sgn}A_{0}(l)-\operatorname{sgn}D}{8}-k\frac{\xi_{l,H,\tilde{H}}^{2}}{2}\Big)$, which is indeed an 8th root of unity by Corollary \ref{veccomp} again. This proves the proposition.
\end{proof}

Using Proposition \ref{formbHJ}, the proofs of Theorems \ref{genact} and \ref{qisotriv} establish the following result.
\begin{thm}
Let $H$, $J$, $\eta$, and $\lambda$ be as in Proposition \ref{formbHJ}, and assume that $J$ is isotropic and that not both of $\eta$ and $\lambda$ are in $J^{\perp}$. Take $M\in\operatorname{Mp}_{2}(\mathbb{Z})$, set $v=\binom{\eta}{\lambda}$, and write the associated vector $\mathfrak{b}_{\eta,\lambda}^{H,J}$ from Equation \eqref{bHJvec} as $\mathfrak{b}_{v}^{H,J}$. Then the action of $\rho_{D}(M)$ takes $\mathfrak{b}_{v}^{H,J}$ to $\varepsilon_{J}(M,v)\mathbf{e}\big(Q(M,v)\big)\mathfrak{b}_{M*v}^{H,J}$, where $Q(M,v)$ is defined in Equation \eqref{QMv}, $M*v$ is defined via Equation \eqref{twistact} (with $\xi_{l,H,\tilde{H}}$ replacing $\xi_{H}$), and $\varepsilon_{J}(M,v)$ is an 8th root of unity. \label{arrowJperp}
\end{thm}

\begin{rmk}
We need that $J$ be isotropic in Theorem \ref{arrowJperp}, in order for the vector $\xi_{H,\tilde{H}}$ to be in $J^{\perp}$, so that the operations from that theorem preserve the property than not both $\lambda$ and $\eta$ are in $J^{\perp}$. We can then set $B:=J^{\perp}/J$, and get a description of the orthogonal complement, inside $\mathbb{C}[D]$, of the sub-representation $\uparrow_{J}\mathbb{C}[B]$ from Corollary \ref{arrow}. Moreover, using Corollary \ref{basisTl} we can easily describe this orthogonal complement using an orthonormal basis. \label{Jperpiso}
\end{rmk}

\begin{rmk}
Note that unlike in Remark \ref{discqiso}, the parameter $\varepsilon_{J}(M,v)$ from Theorem \ref{arrowJperp} is no longer a character of $\operatorname{Mp}_{2}(\mathbb{Z})$. This is because of its dependence on $v$, as seen in the proof of Proposition \ref{formbHJ}. Note also that since $\varepsilon_{J}(T,v)=1$ for every $v$ (see Proposition \ref{formbHJ}), the root of unity $\varepsilon_{J}(M,v)$ depends only on the lower row of $M$ (as well as the metaplectic sign when $\operatorname{sgn}D$ is odd). Note that for odd $p$ and non-isotropic $H$, where as in Remark \ref{discqiso}, every signature may appear,
the two presentations of $\varepsilon_{J}(S,v)$ when $p$ does not divide $l$ look different: One with even $l$ where the term involving $k$ disappears, and one with odd $k$. If $D$ has odd signature, then in the first case $A_{0}(l)$ is a discriminant of odd order hence even signature, but in the second case $k$ is odd, $A_{0}(l)$ has odd signature, and $\frac{\xi_{l,H,\tilde{H}}^{2}}{2}$ is of order 8. However, when $H$ is isotropic, we can replace $\mathfrak{b}_{\eta,\lambda}^{H,J}=\mathfrak{a}_{\eta,\lambda}^{H}$ when $\lambda \in J^{\perp}$ by $\mathbf{e}(\operatorname{sgn}D/8)\mathfrak{a}_{\eta,\lambda}^{H}$, and then all the roots of unity from Proposition \ref{formbHJ} and Theorem \ref{arrowJperp} will be of order 4. Moreover, when $D$ has even signature these roots of unity will be reduced in this way to signs, and if $|A|$ is an even power of $p$ and either $p$ is odd or $A$ satisfies some signature condition, all these parameters disappear and we get an action like in Theorem \ref{genact}. \label{rootunit}
\end{rmk}

We conclude by remarking that the formula from Equation \eqref{withaut}, and its properties given in Proposition \ref{propsym}, can be extended to the vectors from Equation \eqref{bHJvec} and their properties, once extra assumptions are made on $G$. Specifically we need $G$ to preserve $J$ as well (and then the parameter $l$ from Equation \eqref{withaut} remains unaffected), but also the group $\tilde{H}^{\perp}$, whose definition in Lemma \ref{liftmodH0} involved a choice in some cases, must be preserved. These assumptions are clearly satisfied in the cyclic case considered in Theorem \ref{cycdecom}, but not in general.

\section{Integral Bases for Discriminant Forms \label{GenRes}}

In this section we establish the first main goal of this paper, namely proving that the Weil representation $\rho_{D}$ associated with any discriminant form $D$ can always be defined over a ring of algebraic integers.

\smallskip

We begin with the case where $D$ is the cyclic discriminant form $p^{\pm1}$, where $p$ is some odd prime. The even part of this case was essentially dealt with in \cite{[Wa]}, but we reproduce the proof because some of our previous results shorten it significantly. However, the odd part requires an auxiliary technical lemma.
\begin{lem}
Given positive integers $h$ and $m$, assume that we are given an analytic function $\varphi_{m,h}$ of the variable $\zeta$, such that the Taylor expansion of $\varphi_{m,h}$ at $\zeta=1$ is $\sum_{n=0}^{\infty}P_{h,n}(m)(\zeta-1)^{n}$ in which $P_{h,n}$ is an odd polynomial of degree $2n+2h-1$, and the function $\varphi_{m,h}$ is the constant $\delta_{m,h}$ when $m \leq h$. Fix a third integer $r\geq0$, and define functions $f_{m,h}^{(r)}$ of $\zeta$ as follows: For $h=0$ we set $f_{m,0}^{(r)}$ to be the constant $\binom{m+r}{2r+1}$, and for $h\geq1$ the function $f_{m,h}^{(r)}$ is defined inductively as $f_{m,h-1}^{(r)}(\zeta)-\varphi_{m,h}(\zeta)f_{h,h-1}^{(r)}(\zeta)$. Then the function $f_{m,h}^{(r)}$ is the constant $\binom{m+r}{2r+1}$ wherever $h \leq r$, and vanishes to order at least $h-r$ at $\zeta=1$ in case $h \geq r$. \label{expzeta=1}
\end{lem}

\begin{proof}
The binomial coefficients $\binom{m+j}{2j+1}$ with $j\geq0$ form a basis for the space of odd polynomials in $m$, where the $j$th such expression has degree $2j+1$, it vanishes for $m \leq j$, and it attains 1 on $m=j+1$. The degree bound means that $P_{h,n}(m)$ is spanned by $\binom{m+j}{2j+1}$ for $0 \leq j \leq n+h-1$, and the values for small $m$ mean that the terms with $j \leq h-\delta_{n,0}$ do not appear in $P_{h,n}(m)$, and $P_{h,0}(m)=\binom{m+h-1}{2h-1}$ with the coefficient 1. We can thus write \[\varphi_{m,h}(\zeta)=\binom{m+h-1}{2h-1}-\sum_{n=1}^{\infty}\sum_{j=h}^{n+h-1}\alpha_{n,j}^{(h)}\binom{m+j}{2j+1}(\zeta-1)^{n}\] for some constants $\alpha_{n,j}^{(h)}$ for any $n\geq1$, $h\geq1$, and $h \leq j \leq n+h-1$.

Now, the constant $f_{m,0}^{(r)}=\binom{m+r}{2r+1}$ vanishes for $m \leq r$. Therefore, if we assume that $f_{m,k-1}^{(r)}$ is the constant $\binom{m+r}{2r+1}$ for some $1 \leq h \leq r$ (which is given when $h=1$), then the vanishing of $f_{h,h-1}^{(r)}$ in the definition of $f_{m,h}^{(r)}$ implies that the latter equals the same constant as well. This proves the first assertion.

We now claim that for any $h \geq r$, the function $f_{m,h}^{(r)}$ expands at $\zeta=1$ as $\sum_{s=h-r}^{\infty}\sum_{j=h}^{s+r}\beta_{s,j}^{(h,r)}\binom{m+j}{2j+1}(\zeta-1)^{s}$ for some constants $\beta_{s,j}^{(h,r)}$, which will clearly establish the second assertion. The claim is evident for $h=r$, with the coefficients $\beta_{s,j}^{(r,r)}=\delta_{s,0}\delta_{j,r}$. Assuming that the claim holds for $h-1$ for some $h>r$, the fact that for $j \geq h-1$ the expression $\binom{h+j}{2j+1}$ equals $\delta_{j,h-1}$ reduces $f_{h,h-1}^{(r)}(\zeta)$ to $\sum_{s=h-1-r}^{\infty}\beta_{s,h-1}^{(h-1,r)}(\zeta-1)^{s}$. Substituting these expressions into the formula $f_{m,h-1}^{(r)}(\zeta)-\varphi_{m,h}(\zeta)f_{h,h-1}^{(r)}(\zeta)$, the summands with $j=h-1$ cancel with the part coming from the constant term of $\varphi_{m,h}$, and our expression for $f_{m,h}^{(r)}(\zeta)$ becomes the sum of $\sum_{s=h-r}^{\infty}\sum_{j=h}^{s+r}\beta_{s,j}^{(h-1,r)}\binom{m+j}{2j+1}(\zeta-1)^{s}$ and \[\sum_{l=h-1-r}^{\infty}\beta_{l,h-1}^{(h-1,r)}(\zeta-1)^{l}\times\sum_{n=1}^{\infty}\sum_{j=h}^{n+h-1}\alpha_{n,j}^{(h)}\binom{m+j}{2j+1}(\zeta-1)^{n}.\] But with $s=n+l \geq h-r$ and $n=s-l \leq n-h+1+r$ we get the inequalities $h \leq j \leq n+h-1 \leq s+r$, and our claim (with the second assertion) follows, with $\beta_{s,j}^{(h,r)}=\beta_{s,j}^{(h-1,r)}+\sum_{n=j+1-h}^{s+r-1-h}\alpha_{n,j}^{(h)}\beta_{s-n,h-1}^{(h-1,r)}$. This proves the lemma.
\end{proof}

We shall also need a decomposition of a Vandermonde matrix.
\begin{lem}
Let $M$ be the Vandermonde matrix of some parameters $\{x_{m}\}_{m=1}^{n}$, in the convention in which the first column of $M$ consists of 1's. Then in the presentation of $M$ as $LU$, where $L$ is lower triangular and $U$ is upper triangular and unipotent, the entries of $U$ are polynomials in $\{x_{m}\}_{m=1}^{n}$, and $L$ decomposes further as the following product. Let $N_{h}$ be the lower triangular unipotent matrix with $ij$ entry 1 in case $i=h \leq j$ and $\delta_{ij}$ otherwise, and set $D_{h}$ to be the diagonal matrix with $i$th diagonal entry 1 if $i \leq h$ and $x_{i}-x_{h}$ if $i>h$. Then $L$ is the product $N_{1}D_{1}N_{2}D_{2}...N_{n-1}D_{n-1}$. \label{Vdm}
\end{lem}

\begin{proof}
Let $h_{k}$ denote the complete homogeneous symmetric polynomial of degree $k$ (with $h_{0}=1$ and $h_{k}=0$ for $k<0$). Then a classical result (see, e.g., Theorem 2 of \cite{[Ya]} in the alternative convention, though it was known much earlier) implies that in this decomposition of $M$, the $ij$th entry of $U$ is $h_{j-i}(x_{1},\ldots,x_{i})$, while the $ij$th entry of $L$ is $\prod_{m=1}^{j-1}(x_{i}-x_{m})$ (indeed vanishing when $i<j)$, with the empty product 1 when $j=1$. This proves the first assertion. The second one follow by induction, once one verifies that $L=N_{1}D_{1}\binom{1\ \ 0}{0\ \ \tilde{L}}$ where $\tilde{L}$ is the $L$-matrix of the Vandermonde matrix of $\{x_{m}\}_{m=2}^{n}$. This proves the lemma.
\end{proof}

Next we establish a certain explicit formula.
\begin{lem}
Take some parameter $\zeta$ and some $k\geq1$, set $\varepsilon_{m}:=\frac{\zeta^{2km}-\zeta^{-2km}}{\zeta^{2m}-\zeta^{-2m}}$, and consider them to be the entries of a column vector $\varepsilon$. Let $N_{h}$ and $D_{h}$ be the matrices from Lemma \ref{Vdm}, where $x_{m}=\zeta^{m^{2}}$. For every $m$ and $h$ we define the function $\varphi_{m,h}(\zeta)=\frac{\zeta^{2m}-\zeta^{-2m}}{\zeta^{2h}-\zeta^{-2h}}\prod_{j=1}^{h-1}\frac{\zeta^{m^{2}}-\zeta^{j^{2}}}{\zeta^{h^{2}}-\zeta^{j^{2}}}$, and define $f_{m,h}^{(r)}(\zeta)$ as in Lemma \ref{expzeta=1}. Then for $m>h$ the $m$th entry of $D_{h}^{-1}N_{h}^{-1} \ldots D_{1}^{-1}N_{1}^{-1}\varepsilon$ is \[\frac{(\zeta^{2k}-\zeta^{-2k})\sum_{r=0}^{m-1}f_{m,h}^{(r)}(\zeta)\eta_{k}^{2r}}{(\zeta^{2m}-\zeta^{-2m})\prod_{j=1}^{h}(\zeta^{m^{2}}-\zeta^{j^{2}})},\qquad\mathrm{where}\qquad\eta_{k}:=\zeta^{k}-\zeta^{-k}.\] Moreover, the $m$th entry of $L^{-1}\varepsilon$ is given by the same expression, with $h=m-1$. \label{withflhr}
\end{lem}

\begin{proof}
We establish the result by induction, where for $h=0$ we need to express $\varepsilon_{m}$ in a more convenient manner. Recalling that $\frac{X^{n+1}-X^{-n-1}}{X-X^{-1}}$ is given by $U_{n}\big(\frac{X+X^{-1}}{2}\big)$ where $U_{n}\big(\frac{y}{2}\big)=\sum_{s=0}^{\lfloor n/2 \rfloor}(-1)^{s}\binom{n-s}{s}y^{n-2s}$ is the \emph{Chebyshev polynomial of the second kind} (with $\lfloor x \rfloor$ being the lower integral function), we can write $\varepsilon_{m}$ as $\frac{\zeta^{2k}-\zeta^{-2k}}{\zeta^{2m}-\zeta^{-2m}}U_{m-1}\big(\frac{\zeta^{2k}+\zeta^{-2k}}{2}\big)$. As the argument of $U_{m-1}$ here is $\frac{\eta_{k}^{2}}{2}+1$, the formula $\sum_{r=0}^{n}\binom{n+r+1}{2r+1}(2y-2)^{r}$ for $U_{n}(y)$ now allows us to express $\varepsilon_{m}$ as $\frac{\zeta^{2k}-\zeta^{-2k}}{\zeta^{2m}-\zeta^{-2m}}\sum_{r=0}^{m-1}\binom{m+r}{2r+1}\eta_{k}^{2r}$, which is the assertion for $h=0$ by the definition of $f_{m,0}^{(r)}$ in Lemma \ref{expzeta=1} and its vanishing for $m \leq r$.

Take now $h\geq1$, and assume that the $m$th entry of $D_{h-1}^{-1}N_{h-1}^{-1} \ldots D_{1}^{-1}N_{1}^{-1}\varepsilon$ is expressed, for $m>h$, by our formula with $h-1$. The action of $N_{h}^{-1}$, which is defined like $N_{h}$ but with $hj$-entry $-1$ for $j>h$, subtracts from it the same term with $m=h$. As $\varphi_{m,h}(\zeta)$ is the quotient of the denominators, we indeed obtain the recursive definition of $f_{m,h}^{(r)}(\zeta)$ from Lemma \ref{expzeta=1}, and the action of $D_{h}^{-1}$ divides by the remaining expression in the denominator. This establishes the formula with $h$, and the for $L^{-1}\varepsilon$, with $L$ decomposed as in Lemma \ref{Vdm}, we just note that the $m$th entry is not affected by the matrices with index $m$ and larger, and thus preserves the value that it attains for $h=m-1$. This proves the lemma.
\end{proof}

The application that we shall need is the following one.
\begin{prop}
Let $\zeta$ be a non-trivial root of unity of prime order $p$, and take some $1 \leq k \leq p-1$. Let $c_{l}$ with $0 \leq l\leq\frac{p-3}{2}$ be the elements of $\mathbb{Q}(\zeta)$ such that $\sum_{l=0}^{(p-3)/2}(\zeta^{2m}-\zeta^{-2m})\zeta^{lm^{2}}c_{l}=\zeta^{2km}-\zeta^{-2km}$ for every $1 \leq m\leq\frac{p-1}{2}$. Then $c_{l}$ lies in $\mathbb{Z}[\zeta]$ for every $0 \leq l\leq\frac{p-3}{2}$. \label{inintbasis}
\end{prop}

\begin{proof}
Simple division shows that the vector $c$ of our elements is related to $\varepsilon$ from Lemma \ref{withflhr} via $Mc=\varepsilon$, where $M$ is the Vandermonde matrix of the numbers $\zeta_{m^{2}}$ for $1 \leq m\leq\frac{p-1}{2}$, in the convention from Lemma \ref{Vdm}. We thus need the integrality of the entries of $M^{-1}\varepsilon=U^{-1}L^{-1}\varepsilon$, and since the first assertion of that lemma implies that both $U$ and $U^{-1}$ have entries from $\mathbb{Z}[\zeta]$, the desired integrality is equivalent to that of $L^{-1}\varepsilon$. We thus need the integrality of the expression from Lemma \ref{withflhr} with our $\zeta$.

Now, it is clear that the formula in question consists of quotients of differences of the form $\zeta^{a}-\zeta^{b}$. Recall that every non-zero such difference generates the same prime ideal $\mathfrak{p}$ in $\mathbb{Z}[\zeta]$, and quotients between such differences are units in $\mathbb{Z}[\zeta]$. Moreover, the value obtained by substituting $\zeta=1$ in such a unit is non-zero. It therefore follows that the power of $\mathfrak{p}$ dividing any polynomial in $\zeta$ is at least the order of that expression at $\zeta=1$ (this is not equality because the are relations among powers of $\zeta$), but when this polynomial is a product of non-vanishing differences, the power of $\mathfrak{p}$ dividing it is precisely the number of multipliers. In particular, an expression like that from Lemma \ref{withflhr} is integral in case it becomes finite at $\zeta=1$ as a function of $\zeta$.

Next we verify that the functions $\varphi_{m,h}$ from that proposition satisfy the conditions from Lemma \ref{expzeta=1}. The equality $\varphi_{m,h}(\zeta)=\delta_{m,h}$ for $m \leq h$ is clear. Moreover, the $t$th term in the Taylor expansion of each normalized multiplier $\frac{\zeta^{m^{2}}-\zeta^{j^{2}}}{\zeta-1}$ at $\zeta=1$ is an even polynomial of degree $2t+2$ in $m$, while for $\frac{\zeta^{2m}-\zeta^{-2m}}{\zeta-1}$ it is odd of degree $2t+1$. Since the Taylor expansions of the corresponding terms in the denominator of $\varphi_{m,h}(\zeta)$ are independent of $m$, we deduce that the $n$th term in the Taylor expansion of $\varphi_{m,h}(\zeta)$ is indeed an odd polynomial in $m$, of degree $2n+2h-1$.

Finally, the denominator in Lemma \eqref{withflhr} has order $h+1$ at $\zeta=1$. On the other hand, for every $r$ we have an order of $2r$ from $\eta_{k}^{2r}$, Lemma \ref{expzeta=1} implies that $f_{m,h}^{(r)}(\zeta)$ has order at least $h-r$ (this is clear also if $r>h$), and the external multiplier contributes another 1 to the order. Therefore this summand has order at least $h+r+1 \geq h+1$ at $\zeta=1$, meaning that the total quotient is finite at $\zeta=1$, hence lies in $\mathbb{Z}[\zeta]$, for every $m$ and $h$ as desired. This proves the proposition.
\end{proof}

\smallskip

We can now prove the basic case $D=p^{\pm1}$ of our main result.
\begin{lem}
The Weil representation $\rho_{D}$ associated with the discriminant form $D=p^{\pm1}$ has a basis in which the action of $\rho_{D}$ is integral. \label{irrodd}
\end{lem}

\begin{proof}
We decompose $\rho_{D}$ into the even and odd parts, and consider each one separately. We also choose a generator $\eta$ of $D$, write every index $m\eta$ for $m\in\mathbb{F}_{p}$ as simply $m$, and set $\zeta:=\mathbf{e}\big(\frac{\eta^{2}}{2}\big)=\mathbf{e}\big(\frac{t}{p}\big)$ for some $t\in\mathbb{F}_{p}^{\times}$ with $\big(\frac{t}{p}\big)=\pm1$ is the sign for which $D=p^{\pm1}$.

For the even part, consider the vector $\mathfrak{a}_{0,0}^{D,+}=\mathfrak{a}_{0,0}^{D}$. Recall from the proof of Theorem \ref{cycdecom} that the $\pm$ part of $\mathbb{C}[D]$ consists of a basis of eigenvectors for $\rho_{D}(T)$ with different values, and all of these appear in $\mathfrak{a}_{0,0}^{D,+}$. Since there are $\frac{p+1}{2}$ such eigenvalues, which are of the form $\zeta^{j^{2}}$ for $t\in\mathbb{F}_{p}^{\times}/(\mathbb{F}_{p}^{\times})^{2}$ determined by the sign $\pm$ of $D$ and $0 \leq j\leq\frac{p-1}{2}$, we deduce that $\big\{\rho_{D}(T^{l})\mathfrak{a}_{0,0}^{D,+}\big\}_{l=0}^{(p-1)/2}$ form a basis for the corresponding sub-representation of $\rho_{D}$. We prove that this basis has the desired property.

Indeed, it suffices to verify this property for $T$ and $S$. For $T$, each vector but the last is taken to the next one, and the statement is clear. On the other hand, the previous paragraph implies that the presentation of the $\rho_{D}(T)$-image of the last vector in terms of the previous ones is based on the coefficients of the characteristic polynomial of $\rho_{D}(T)$. But we have seen that the roots of this polynomial are roots of unity, so that these coefficients are algebraic integers, as desired.

For the action of $S$, Remark \ref{rhoSsym} allows us to apply Proposition \ref{actofS}, with the parameters $\eta=\lambda=\xi_{l,H}=0$, for evaluating the action on the vectors with $l>0$. This determines the corresponding $\rho_{D}(S)$-image as $\mathbf{e}\big(\frac{\operatorname{sgn}D(l)-\operatorname{sgn}D}{8}\big)$ times $\rho_{D}(T^{-k})\mathfrak{a}_{0,0}^{D,+}$ (in fact, the results of \cite{[Sch]}, \cite{[Str]}, \cite{[Ze]}, and others determine this coefficient to be just the Legendre symbol $\big(\frac{l}{p}\big)$), where we can take $-p<k<0$. This vector thus has algebraic integer coefficients in our basis, by what we know about powers of $T$.

It remains to verify that $\rho_{D}(S)\mathfrak{a}_{0,0}^{D,+}$, or equivalently $\mathfrak{e}_{0}$, has integral coefficients in our basis. Writing the latter vector as $\sum_{l=0}^{(p-1)/2}c_{l}\rho_{D}(T)\mathfrak{a}_{0,0}^{D,+}$, and checking the coefficients of $\mathfrak{a}_{0,j}^{\{0\},+}$, we find that $\sum_{l=0}^{(p-1)/2}c_{l}\zeta^{j^{2}l}=0$ for every $1 \leq j\leq\frac{p-1}{2}$, so that the polynomial $\sum_{l=0}^{(p-1)/2}c_{l}X^{l}$ is a multiple of $\prod_{j=1}^{(p-1)/2}(X-\zeta^{j^{2}l})$ (the multiplier is $c:=c_{(p-1)/2}$). It thus suffices to check that this multiplier is an algebraic integer, for which we compare the coefficients of $\mathfrak{e}_{0}$ on both sides and get $1=\sum_{l=0}^{(p-1)/2}c_{l}=\frac{c}{\sqrt{p}}\prod_{j=1}^{(p-1)/2}(1-\zeta^{j^{2}l})$. But each multiplier generates the prime ideal $\mathfrak{p}$ from the proof of Proposition \ref{inintbasis}, and it is known that if $\pm p(\mathrm{mod\ }4)$ then $\sqrt{\pm p}\in\mathbb{Z}[\zeta]$ and it generates the ideal $\mathfrak{p}^{(p-1)/2}$. It follows that the product and $\sqrt{\pm p}$ generate the same ideal, and thus $c\in\mathbb{Z}[\zeta]$ (and is a unit there). This establishes the desired property of our basis of the even part.

In the odd part there is no vector $\mathfrak{a}_{0,0}^{D,-}$, so we work with $\mathfrak{a}_{1,0}^{D,-}$ instead (representing a generator of $D$). Here $\rho_{D}(T^{l})\mathfrak{a}_{1,0}^{D,-}$ with $0 \leq l\leq\frac{p-3}{2}$ is a basis, since the eigenvalues of $\rho_{D}(T)$ are $\zeta^{j^{2}}$ with $1 \leq j\leq\frac{p-1}{2}$, and the action of $T$ involves integral coefficients by the same argument.

Next, we claim that any other possible generator, namely $\mathfrak{a}_{k,0}^{D,-}$ for some non-zero $k\in\mathbb{F_{p}}$, is generated by our basis with integral coefficients. Indeed, using that basis for the anti-symmetric part of $\rho_{D}$, there are coefficients $c_{l}$ with $0 \leq l\leq\frac{p-3}{2}$ such that $\mathfrak{a}_{k\eta,0}^{D,-}=\sum_{l=0}^{(p-3)/2}c_{l}\rho_{D}(T^{l})\mathfrak{a}_{1,0}^{D,-}$. But Equation \eqref{symvec} gives \[\mathfrak{a}_{1,0}^{D,-}=\sum_{m=0}^{p-1}\frac{\zeta^{2m}-\zeta^{-2m}}{\sqrt{2p}}\mathfrak{e}_{m\eta}=\sum_{m=1}^{(p-1)/2}\frac{\zeta^{2m}-\zeta^{-2m}}{\sqrt{p}}\mathfrak{a}_{0,m}^{\{0\},-},\] $\rho_{D}(T^{l})$ multiplies the $m$th summand by $\zeta^{lm^{2}}$, and in $\mathfrak{a}_{k,0}^{D,-}$ we replace $\zeta^{\pm2m}$ by $\zeta^{\pm2km}$. Therefore our coefficients are the ones from Proposition \ref{inintbasis}, yielding their integrality.

Combining the action of $T$, it thus follows that $\rho_{D}(T^{-k})\mathfrak{a}_{k\eta,0}^{D,-}$ is also spanned by our basis with integral coefficients, and since for $l>0$ a similar application of Proposition \ref{actofS} (using Remark \ref{rhoSsym} again) identifies $\rho_{D}(S)\rho_{D}(T^{l})\mathfrak{a}_{\eta,0}^{D,-}$ with this vector (for the appropriate $k$) up to a root of unity, the integrality of the action of $S$ on these vectors is established.

Finally, for $\rho_{D}(S)\mathfrak{a}_{1,0}^{D,-}$, which is a root of unity times $\mathfrak{a}_{0,1}^{\{0\},-}$ (via Proposition \ref{propsym}), writing it again as $\sum_{l=0}^{(p-3)/2}c_{l}\rho_{D}(T^{l})\mathfrak{a}_{1,0}^{D,-}$ and expanding using the $\mathfrak{a}_{0,m}^{\{0\},-}$'s, the coefficients become those of $c\prod_{j=2}^{(p-1)/2}(X-\zeta^{j^{2}})$ (now with $c:=c_{(p-3)/2}$), and since $c$ now satisfies $1=\frac{c}{\sqrt{p}}\prod_{j=2}^{(p-1)/2}(\zeta-\zeta^{j^{2}})$, we deduce that $c$ is in $\mathbb{Z}[\zeta]$ (and now generates $\mathfrak{p}$), and $S$ acts integrally in our basis as well. This proves the lemma.
\end{proof}
In fact, one can prove Lemma \ref{expzeta=1} with even polynomials and $h\geq0$, with the initial condition where $f_{m,0}^{(r)}$ equals 2 for $r=0$ and $\frac{m}{r}\binom{m-1+r}{2r-1}$ when $r>0$. The one can replace $\varepsilon_{m}$ by $\frac{\zeta^{2km}+\zeta^{-2km}}{\zeta^{2m}+\zeta^{-2m}}$ and set  $\varphi_{m,h}(\zeta)=\frac{\zeta^{2m}+\zeta^{-2m}}{\zeta^{2h}+\zeta^{-2h}}\prod_{j=0}^{h-1}\frac{\zeta^{m^{2}}-\zeta^{j^{2}}}{\zeta^{h^{2}}-\zeta^{j^{2}}}$ in Lemma \ref{withflhr} (also for $h=0$). This proves Proposition \ref{inintbasis} with $0 \leq m\leq\frac{p-1}{2}$, a plus sign between the powers of $\zeta$, and $0 \leq l\leq\frac{p-1}{2}$, which implies that $\mathfrak{a}_{k,0}^{D,+}$ is integrally spanned by $\rho_{D}(T^{l})\mathfrak{a}_{1,0}^{D,+}$ with $0 \leq l\leq\frac{p-1}{2}$. Then the proof of the even part of Lemma \ref{irrodd} will also work with the latter basis, though with the one arising from $\mathfrak{a}_{0,0}^{D,+}$ the proof from \cite{[Wa]} is simpler.

\smallskip

We can now treat more general anisotropic discriminant forms.
\begin{prop}
Let $D$ be an anisotropic discriminant form of prime power cardinality. Then $\mathbb{C}[D]$ admits a basis in which $\rho_{D}$ acts with algebraic integer coefficients. \label{aniso}
\end{prop}

\begin{proof}
The anisotropic discriminant forms that we have to consider were determined in, e.g., \cite{[Zh]}. When the underlying prime $p$ is odd, there are only the ones of degree $p$, and one that is 2-dimensional over $\mathbb{F}_{p}$. The former were dealt with in Lemma \ref{irrodd}, and the latter can be presented as an orthogonal direct sum, so that the representation is a tensor product of ones that are also covered by that lemma.

When $p=2$, the discriminant form $D$ is $2^{+1}_{\pm1}$, $2^{+2}_{\pm2}$, $2^{+3}_{\pm3}$, $2^{-2}_{II}$ or one of $4^{\pm1}_{t}$ and $2^{+1}_{\pm1}4^{\pm1}_{t}$ for odd $t$. Since $2^{-2}_{II}$ and $4^{\pm1}_{t}$ admit self-dual quasi-isotropic subgroups, the assertion for them follows from Theorem \ref{qisotriv}. Moreover, using the same direct sum argument, it suffices to consider the case where $D$ has order 2.

Thus assume that $D=2^{+1}_{\pm1}$, with elements 0 and 1, and as in Lemma \ref{irrodd} we take the base consisting of $\mathfrak{a}_{0,0}^{D}=\frac{1}{\sqrt{2}}(\mathfrak{e}_{0}+\mathfrak{e}_{1})$ and $\rho_{D}(T)\mathfrak{a}_{0,0}^{D}=\frac{1}{\sqrt{2}}(\mathfrak{e}_{0}\pm i\mathfrak{e}_{1})$. It is clear that $\rho_{D}(T)$ takes the first basis vector to the second, and as the image $\rho_{D}(T)^{2}\mathfrak{a}_{0,0}^{D}=\frac{1}{\sqrt{2}}(\mathfrak{e}_{0}-\mathfrak{e}_{1})$ of the second one is represented in our basis as $(1\pm i)\rho_{D}(T)\mathfrak{a}_{0,0}^{D}\mp i^{t}\mathfrak{a}_{0,0}^{D}$, the statement for $T$ is established.

As for $\rho_{D}(S)$, for the second vector we apply Proposition \ref{actofS}, with $k=-3$, and get $\frac{1}{\sqrt{2}}(\mathfrak{e}_{0}\mp i\mathfrak{e}_{1})$ (the root of unity can be evaluated directly to be just 1), which is the combination $(1\mp i)\mathfrak{a}_{0,0}^{D}\pm i\rho_{D}(T)\mathfrak{a}_{0,0}^{D}$ of our basis vectors. Finally, Lemma \ref{rhoDaH} and the evaluation of the root of unity $\mathbf{e}\big(-\frac{\operatorname{sgn}D}{8}\big)$ shows that $\rho_{D}(S)$ takes the first basis vector to $\frac{1\mp i}{\sqrt{2}}\mathfrak{e}_{0}$. As the presentation of this vector in the basis is $\rho_{D}(T)\mathfrak{a}_{0,0}^{D}\mp i^{t}\mathfrak{a}_{0,0}^{D}$, this proves assertion for $S$ as well. This completes the proof of the proposition.
\end{proof}

Combining all the results we have gathered, we can prove the main theorem of this paper.
\begin{thm}
Let $D$ be any discriminant form. Then there is a basis for the space $\mathbb{C}[D]$ on which the associated Weil representation $\rho_{D}$ operates with coefficients that are algebraic integers. \label{fin}
\end{thm}

\begin{proof}
Decomposing $\rho_{D}$ as the tensor product of the Weil representations arising from these $p$-adic parts, it suffices to consider the case where $D$ has prime power level. In this case we apply induction on the cardinality of the maximal isotropic subgroup of $D$, where the case in which such $H$ is trivial is proved in Proposition \ref{aniso}. Consider now the case where $D$ contains isotropic vectors, and let $H$ be a maximal quasi-isotropic subgroup of $D$. Observing that the anisotropic discriminant forms $2^{-2}_{II}$, $4^{\pm1}_{t}$, and $2^{+1}_{\pm1}4^{\pm1}_{t}$ from Proposition \ref{aniso} contain quasi-isotropic elements, we deduce that the quotient $A:=H^{\perp}/H$ is either trivial or of prime exponent $p$.

Now, when $A$ is trivial the result follows directly from the integrality of the coefficients from Theorems \ref{genact} and \ref{qisotriv}, combined with the fact that the extra roots of unity obtained by Lemma \ref{ortho} when one takes a basis as in that lemma are also integral. Assuming now that $A$ is non-trivial, and so is $H$, we take an isotropic subgroup $J$ of $H$, of cardinality $p$, set $B:=J^{\perp}/J$, and recall that Corollary \ref{arrow} embeds the representation $\mathbb{C}[B]$ into $\mathbb{C}[D]$ via $\uparrow_{J}$. Then Remark \ref{Jperpiso} explains, via the same argument, how Theorem \ref{arrowJperp} yields a basis for the orthogonal complement of $\uparrow_{J}\mathbb{C}[B]$ inside $\mathbb{C}[D]$ in which the coefficients are algebraic integers, and it remains to obtain such a basis for $\uparrow_{J}\mathbb{C}[B]$ itself. But as Corollary \ref{arrow} shows that $\uparrow_{J}$ is an embedding of representations, and a maximal isotropic subgroup in $B$ is $H/J$ which is of smaller cardinality, this part is covered by the induction hypothesis. This proves the theorem.
\end{proof}

We conclude this section by remarking that while Theorems \ref{genact} and \ref{qisotriv} make the formulae for the action of every element in that basis very explicit, we can no longer do that in the general case considered in Theorem \ref{fin}. Indeed, when one $p$-part involves a non-trivial anisotropic quotient, the formulae resulting from Lemma \ref{irrodd} and Proposition \ref{aniso} become complicated when one tries to consider a general element of $\operatorname{Mp}_{2}(\mathbb{Z})$, and Remark \ref{rootunit} shows that the extra parameter from Theorem \ref{arrowJperp}, in the corresponding complement, can also be hard to evaluate.

\section{Invariant Vectors in Weil Representations \label{InvVect}}

Let $D$ be an arbitrary discriminant form, with the associated Weil representation $\rho_{D}$ of $\operatorname{Mp}_{2}(\mathbb{Z})$ on the space $\mathbb{C}[D]$. A natural question would be to determine the subspace $\mathbb{C}[D]^{\mathrm{inv}}$ of $\mathbb{C}[D]$ on which $\rho_{D}$ operates trivially, and investigate its properties. The paper \cite{[ES]} established one property, which is related to the integrality questions considered in this paper, by showing that $\mathbb{C}[D]^{\mathrm{inv}}$ can always be defined over $\mathbb{Z}$.

\smallskip

Two well-known and basic conditions that make $\mathbb{C}[D]^{\mathrm{inv}}$ trivial are as follows.
\begin{lem}
If $\operatorname{sgn}D$ is odd then $\mathbb{C}[D]^{\mathrm{inv}}=\{0\}$. In addition, let $D_{\mathrm{iso}}$ denote the set of isotropic elements of $D$. Then if $D \setminus D_{\mathrm{iso}}$ surjects onto the quotient $D/D_{\mathrm{iso}}^{\perp}$ then $\mathbb{C}[D]^{\mathrm{inv}}=\{0\}$ as well. \label{invtriv}
\end{lem}

\begin{proof}
The first assertion follows immediately from the fact that $Z^{2}\in\operatorname{Mp}_{2}(\mathbb{Z})$ always acts as the scalar $(-1)^{\operatorname{sgn}D}$. For the second one, the action of $\rho_{D}(T)$ implies that $\mathbb{C}[D]^{\mathrm{inv}}\subseteq\bigoplus_{\gamma \in D_{\mathrm{iso}}}\mathbb{C}\mathfrak{e}_{\gamma}$, which we embed into $\bigoplus_{\gamma \in D^{\mathrm{iso}}}\mathbb{C}\mathfrak{e}_{\gamma}$ where $D^{\mathrm{iso}}$ denotes the subgroup of $D$ that is generated by $D_{\mathrm{iso}}$. Since $\rho_{D}(S)$ must preserve $\mathbb{C}[D]^{\mathrm{inv}}$, we deduce that $\mathbb{C}[D]^{\mathrm{inv}}$ is contained in the subspace of $\bigoplus_{\gamma \in D^{\mathrm{iso}}}\mathbb{C}\mathfrak{e}_{\gamma}$ that is taken to $\bigoplus_{\gamma \in D_{\mathrm{iso}}}\mathbb{C}\mathfrak{e}_{\gamma}$ by $\rho_{D}(S)$.

The formula for $\rho_{D}(S)$ from Equation \eqref{Weildef} thus shows that if $\sum_{\gamma \in D^{\mathrm{iso}}}c_{\gamma}\mathfrak{e}_{\gamma}$ is in the latter space then $\sum_{\gamma \in D^{\mathrm{iso}}}c_{\gamma}\mathbf{e}\big(-(\gamma,\delta)\big)=0$ for every $\delta \notin D_{\mathrm{iso}}$. But elements in the same coset modulo $D_{\mathrm{iso}}^{\perp}$ give the same relation, while a collection of different cosets yields linearly independent relations. Thus the surjectivity condition amounts to the elements of $D/D_{\mathrm{iso}}^{\perp}$ producing all the $|D/D_{\mathrm{iso}}^{\perp}|=|D^{\mathrm{iso}}|$ linearly independent conditions, meaning that the space in question is $\{0\}$, and thus so is $\mathbb{C}[D]^{\mathrm{inv}}$. This proves the lemma.
\end{proof}

\begin{rmk}
If $D_{\mathrm{iso}}$ is a subgroup of $D$ such that the inclusion in $D_{\mathrm{iso}}^{\perp}$ is strict (e.g., when $D$ is anisotropic and non-trivial, or cyclic of non-square order), then the surjectivity condition from Lemma \ref{invtriv} holds, and we get $\mathbb{C}[D]^{\mathrm{inv}}=\{0\}$. Moreover, the signature condition implies that wherever $\mathbb{C}[D]^{\mathrm{inv}}\neq\{0\}$, the Weil representation $\rho_{D}$ is a representation of $\operatorname{SL}_{2}(\mathbb{Z})$, without double covers. \label{condinv}
\end{rmk}

\smallskip

In contrast to Lemma \ref{invtriv} and Remark \ref{condinv}, when $D$ contains a self-dual isotropic subgroup $H$, the representation $\mathbb{C}[A]$ for $A=H^{\perp}/H$ is trivial, and its image under the operator from Corollary \ref{arrow} lies in $\mathbb{C}[D]^{\mathrm{inv}}$. Moreover, it is known that in this case $\mathbb{C}[D]^{\mathrm{inv}}$ is spanned by these images for the various self-dual isotropic subgroups of $D$. This was proved as Theorem 5.5.7 of \cite{[NRS]} in the language of codes and mentioned as Theorem 1 of \cite{[ES]}, and the proof was translated to our terminology in Theorem 4.4 of \cite{[Bi]}. Note that some discriminant forms need not contain self-dual isotropic subgroups at all, but still admit invariant vectors---see, e.g., Theorem \ref{Fpvs} below.

These results, however, do not say much about the dimension of $\mathbb{C}[D]^{\mathrm{inv}}$ (or alternatively the number of such subgroups $H$ and the dimension of relations between their images). We now demonstrate how the explicit formulae from Theorems \ref{genact}, \ref{qisotriv}, and \ref{arrowJperp} can be used for determining the dimension of $\mathbb{C}[D]^{\mathrm{inv}}$ in many cases. We only present two families, one involving self-dual isotropic subgroups in a strong way, and the other not necessarily having ones at all.

\smallskip

For defining the first family, let $G$ be any finite Abelian group, with dual $G^{*}$ (which is thus isomorphic to $G$ as an Abelian group, although not canonically). Then the \emph{generalized hyperbolic plane associated with $G$} is the discriminant form $U_{G}:=G \oplus G^{*}$, where given $\gamma \in G$ and $\phi \in G^{*}=\operatorname{Hom}(G,\mathbb{Z})$, the element $\gamma+\phi \in U_{G}$ satisfies $\frac{\gamma+\phi}{2}=\phi(\gamma)$. For obtaining the formula for $\dim\mathbb{C}[U_{G}]^{\mathrm{inv}}$ we shall use the following simple lemma. We denote Euler's totient function as usual by $\varphi(n):=\prod_{p|n}\big(1-\frac{1}{p}\big)$.
\begin{lem}
Take two numbers $n$ and $m$, with $m|n$, consider the Abelian group $\mathbb{Z}/n\mathbb{Z}\oplus\mathbb{Z}/m\mathbb{Z}$, and let $X_{n,m}$ be the set of pairs of elements in that group that generate it. Then the group $\operatorname{SL}_{2}(\mathbb{Z}/n\mathbb{Z})$ acts on $X_{n,m}$ with $\varphi(m)$ orbits. \label{SLnongen}
\end{lem}

\begin{proof}
We write our group as a quotient of $\mathbb{Z}/n\mathbb{Z}\oplus\mathbb{Z}/n\mathbb{Z}\cong\operatorname{M}_{2}(\mathbb{Z}/n\mathbb{Z})$. Then $X_{n,m}$ is the quotient of $X_{n,n}\cong\operatorname{GL}_{2}(\mathbb{Z}/n\mathbb{Z})$, on which $\operatorname{GL}_{2}(\mathbb{Z}/n\mathbb{Z})$ acts transitively. Now, since $X_{n,n}$ can be presented as $\bigcup_{h\in(\mathbb{Z}/n\mathbb{Z})^{\times}}\operatorname{SL}_{2}(\mathbb{Z}/n\mathbb{Z})\binom{1\ \ 0}{0\ \ h}$, and we project this onto $X_{n,m}$. This gives $\varphi(m)$ orbits with $h\in(\mathbb{Z}/m\mathbb{Z})^{\times}$, and since they are clearly distinct after projecting onto $X_{m,m}$ and working modulo $m$, they are distinct in $X_{n,m}$. This proves the lemma.
\end{proof}

The dimension of $\mathbb{C}[D]^{\mathrm{inv}}$ for $D=U_{G}$ for such $G$ can now be determined.
\begin{thm}
Take $D$ to be the generalized hyperbolic plane $U_{G}$ that is associated with the finite Abelian group $G$. For every pair of positive integers $m$ and $n$ with $m|n$, denote by $S_{n,m}(G)$ the number of subgroups of $G$ that are isomorphic to $\mathbb{Z}/n\mathbb{Z}\oplus\mathbb{Z}/m\mathbb{Z}$. Then we have $\dim\mathbb{C}[D]^{\mathrm{inv}}=\sum_{0<m|n}S_{n,m}(G)\varphi(m)$. \label{hypplane}
\end{thm}

\begin{proof}
Since $\operatorname{sgn}D=0$ (because $G$ is a self-dual isotropic subgroup of $D:=U_{G}$, and so is $G^{*}$), and $D$ has level $N:=\exp(G)$, it is known (see, e.g., \cite{[Str]} or \cite{[Ze]}) that $\rho_{D}$ is essentially a representation of the finite group $\operatorname{SL}_{2}(\mathbb{Z}/N\mathbb{Z})$, whose cardinality is $\Delta_{N}:=N^{3}\prod_{p|N}\big(1-\frac{1}{p^{2}}\big)$. Relating $\rho_{D}$ with the trivial representation of this finite group via the classical formula of Frobenius thus expresses $\dim\mathbb{C}[D]^{\mathrm{inv}}$ as $\frac{1}{\Delta_{N}}\sum_{M\in\operatorname{SL}_{2}(\mathbb{Z}/N\mathbb{Z})}\operatorname{Tr}\rho_{D}(M)$.

We take a basis for $\mathbb{C}[D]$ as in Lemma \ref{ortho}, where for the set of representatives $\mathfrak{R}$ for $D/H$ (also for $\eta$) we can take $G^{*}$. Theorem \ref{genact} expresses the action of any element of $M\in\operatorname{SL}_{2}(\mathbb{Z}/N\mathbb{Z})$ in this basis, where for every such $v$ in the given basis, $Mv$ also lies there because both parameters are from the subgroup $G^{*}$. Moreover, as $G^{*}$ is also isotropic, the expression from Equation \eqref{QMv} is trivial for every $M$ and $v$. Thus the contribution of the vector $\mathfrak{a}_{v}^{H}$ to $\operatorname{Tr}\rho_{D}(M)$ is 1 when $Mv=v$, and 0 otherwise. This allows us to write $\dim\mathbb{C}[D]^{\mathrm{inv}}=\sum_{v\in(G^{*})^{2}}\frac{|\operatorname{St}(v)|}{\Delta_{N}}$, where $\operatorname{St}(v)$ is the stabilizer of $v$ under the action of $\operatorname{SL}_{2}(\mathbb{Z}/N\mathbb{Z})$ induced from Remark \ref{actindvec}.

We partition the latter sum according to the orbits, in $(G^{*})^{2}$, of the action of $\operatorname{SL}_{2}(\mathbb{Z}/N\mathbb{Z})$. Since all the elements in the same orbit have the same stabilizer, they give the same contribution. But as the Orbit-Stabilizer Theorem determines the size of the orbit of $v$ to be $\frac{\Delta_{N}}{|\operatorname{St}(v)|}$, we deduce that $\dim\mathbb{C}[D]^{\mathrm{inv}}$ is just the number of these orbits.

Now, given $v=\binom{\eta}{\lambda}$ with $\eta$ and $\lambda$ in $G^{*}$, it is easy to check that for every element of the orbit of $v$, the corresponding pair of elements of $G^{*}$ generates the same subgroup as $\eta$ and $\lambda$. We thus partition our orbits according to these subgroups, and since only groups that can be generated by two elements are involved, each such group is isomorphic to $\mathbb{Z}/n\mathbb{Z}\oplus\mathbb{Z}/m\mathbb{Z}$ for $n$ and $m$ with $m|n$ (note that $m$, and also $n$, can be 1 when this group is cyclic). On an orbit associated with such a subgroup of $G$, the group $\operatorname{SL}_{2}(\mathbb{Z}/N\mathbb{Z})$ acts through its quotient $\operatorname{SL}_{2}(\mathbb{Z}/n\mathbb{Z})$ (as $n|N$).

But the fact that the action is on sets of generators yields, via Lemma \ref{SLnongen}, that each such subgroup contributes $\varphi(m)$ orbits. Since for such $n$ and $m$ there are $S_{n,m}(G)$ such subgroups, the total number of orbits is given by the asserted formula. This proves the theorem.
\end{proof}
As a special case, we deduce a quick proof of Lemma 3.2 of \cite{[Ye]} (see also Corollary 4.16 of \cite{[Bi]}). We denote by $\sigma_{0}(N)$ the number of positive divisors of the integer $N$.
\begin{cor}
Let $U_{N}$ denote the discriminant form associated with the lattice $II_{1,1}(N)$ for some integer $N$, where $e$ and $f$ are the $U_{N}$-images of the natural, isotropic generators of $II_{1,1}(N)^{*}$. For every $d|N$, set $\mathfrak{a}_{d}:=\sum_{d|a}\sum_{\frac{N}{d}|b}\mathfrak{e}_{ae+bf}$. Then $\{\mathfrak{a}_{d}\}_{d|N}$ form a basis for $\mathbb{C}[U_{N}]^{\mathrm{inv}}$. \label{hypcyc}
\end{cor}

\begin{proof}
The discriminant form $D:=U_{N}$ is generalized hyperbolic plane associated with a cyclic group $G$ of order $N$. We shall denote by $e$ a generator of $G$, and by $f$ a generator of $G^{*}$. Applying Theorem \ref{hypplane}, we have $S_{n,m}=1$ when $m=1$ and $n|N$ and 0 otherwise, implying that $\dim\mathbb{C}[D]^{\mathrm{inv}}=\sigma_{0}(N)$.

Now, Corollary \ref{arrow} shows that for every self-dual isotropic subgroup $H$ of $D$, which is of size $N$, the image of the trivial representation under $\uparrow_{H}$ gives an 1-dimensional subspace of $\mathbb{C}[D]^{\mathrm{inv}}$. For $d|N$ we denote by $H_{d}$ the subgroup generated by $de$ and $\frac{N}{d}f$, and it is easy to verify that $\{H_{d}\}_{d|N}$ are precisely the self-dual isotropic subgroup of $D$. Moreover, It is clear that $\mathfrak{a}_{d}$ is $\sqrt{N}$ times $\mathfrak{a}_{0,0}^{H_{d}}$, which thus generates the image of $\uparrow_{H_{d}}$, and there are $\sigma_{0}(N)$ such vectors.

But given a divisor $D|N$, the coefficient in front of $\mathfrak{e}_{D}$ in a linear combination of the sort $\sum_{d|N}c_{d}\mathfrak{a}_{d}$ is just $\sum_{d|D}c_{d}$, so that if $d$ is the minimal divisor of $N$ with $c_{d}\neq0$ then $\mathfrak{e}_{d}$ appears with a non-zero coefficient. It follows that $\{\mathfrak{a}_{d}\}_{d|N}$ are $\sigma_{0}(N)$ linearly independent vectors in a $\sigma_{0}(N)$-dimensional space, which thus form a basis. This proves the corollary.
\end{proof}

\smallskip

In contrast with Corollary \ref{hypcyc}, when $G$ is not cyclic, the images arising from self-dual are no longer linearly independent. Section 4.3 of \cite{[Bi]} examines the case where $G \cong G_{N,M}=(\mathbb{Z}/N\mathbb{Z})\times(\mathbb{Z}/M\mathbb{Z})$ where $M$ divides $N$ is generated by two elements (then $U_{G}$ is $D_{N,M}$ in the notation of \cite{[Bi]}). To illustrate the type of expressions resulting in this case, we define $\psi(n):=n\prod_{p|n}\big(1+\frac{1}{p}\big)$. Then, since $S_{k,1}(G_{k,d})$ equals $\frac{\psi(d)}{\psi(k/d)}$ wherever $d$ divides $k$, and an subgroup counted in $X_{n,m}(G_{N,M})$ is contained in $G[n]$ (the subgroup of $G$ annihilated by $n$) but must contain $G[m]$, we deduce that $S_{n,m}(G_{N,M})=\frac{\psi(n/m)}{\psi(n/\gcd\{n,M\})}$. Thus Theorem \ref{hypplane} determines $\dim\mathbb{C}[D_{N,M}]$ as \[\sum_{m|M}\sum_{d|\frac{M}{m}}\sum_{\substack{k|\frac{N}{md} \\ \gcd\{k,\frac{M}{md}\}=1}}\frac{\psi(kd)\varphi(m)}{\psi(k)}=\sum_{t|M}\sum_{\substack{k|\frac{N}{t} \\ \gcd\{k,\frac{M}{t}\}=1}}\sum_{d|t}\frac{\psi(kd)\varphi(t/d)}{\psi(k)}.\] In the prime power case, with $N=p^{r}$ and $M=p^{s}$ for $r \geq s$, the latter number becomes \[(r+1-s)(s+1)p^{s}-(r-1-s)sp^{s-1}.\] On the other hand, if $M=p$ is prime and $N_{p}$ is the maximal divisor of $N$ that is prime to $p$ then it equals $(2p-1)\sigma_{0}(N/p)+2\sigma_{0}(N_{p})$ (it also produces $\sigma_{0}(N)$ when $M=1$, as in Corollary \ref{hypcyc}).

Section 4.3 of \cite{[Bi]} also indicates a few constructions, which may serve to count the number of self-dual isotropic subgroups of $D_{N,M}$. We only remark that for $M=p$ a prime, there are $\sigma_{0}(N/p)$ isotropic subgroups $H$ of $D_{N}$ with the property that $K^{\perp}/K \cong D_{p}$, and each one of them is the kernel of a surjective projection onto $D_{p}$ of $2p-2$ self-dual isotropic subgroups of $D_{N,p}$. Since there are also $2\sigma_{0}(N)=2\sigma_{0}(N_{p})+2\sigma_{0}(N/p)$ products of a self-dual isotropic subgroup of $D_{N}$ with one of $D_{p}$, the total number of such subgroups of $D_{N,p}$ is $2p\sigma_{0}(N/p)+2\sigma_{0}(N_{p})$. The $\sigma_{0}(N/p)$ resulting linear relations are in one-to-one correspondence with the groups $K$ mentioned above (the case $N=M=p$, with one relation among $2p+2$ groups, appears in Proposition 4.25 of \cite{[Bi]}). It would be interesting to investigate the combinatorics of the relations arising from more complicated groups $G$.

\smallskip

We now determine the dimension of the space of invariants for another family of discriminant forms.
\begin{thm}
Assume that $D$ is a vector space over $\mathbb{F}_{p}$, and that if $p=2$ then its index is $II$. Let $p^{d}$ be the cardinality of a maximal isotropic subgroup $H$ of $D$, and denote the size of the anisotropic discriminant form $A:=H^{\perp}/H$ by $p^{r}$. Then the dimension of $\mathbb{C}[D]^{\mathrm{inv}}$ is $p^{r}\frac{(p^{d}-1)(p^{d-1}-1)}{p^{2}-1}+\frac{p^{d}-1}{p-1}+\delta_{r,0}$.
\label{Fpvs}
\end{thm}

\begin{proof}
Lemma \ref{invtriv} (or Remark \ref{condinv}) deals with the case where $d=0$ and $r>0$, and if $d=r=0$ then $D$ is trivial and so is $\rho_{D}$. We have thus established the induction basis for working by induction on $d$.

If $d>0$ then $H$ is non-trivial, so we take a cyclic subgroup $J \subseteq H$, and set $B:=J^{\perp}/J$. Then we have an orthogonal decomposition of $\mathbb{C}[D]$ as the direct sum of the sub-representation $\uparrow_{J}\mathbb{C}[B]$ from Corollary \ref{arrow} and its orthogonal complement. The former contributes $\dim\mathbb{C}[B]^{\mathrm{inv}}$, which by the induction hypothesis (with $d-1$ and the same $r$) is $p^{r}\frac{(p^{d-1}-1)(p^{d-2}-1)}{p^{2}-1}+\frac{p^{d-1}-1}{p-1}+\delta_{r,0}$, and for the complement we argue as in the proof of Theorem \ref{hypplane}.

Indeed, the signature is even and the level is $p$ (this is why we need the index to be $II$ when $p=2$), so we view $\rho_{D}$ as a representation of $\operatorname{SL}_{2}(\mathbb{F}_{p})$, of cardinality $p^{3}-p$. Take now a representing set $\mathfrak{V}$ for the vectors $v \in D^{2}\setminus(J^{\perp})^{2}$ modulo the relations from Corollary \ref{basisTl}, and we allow ourselves the abuse of terminology by saying that $M\in\operatorname{SL}_{2}(\mathbb{F}_{p})$ stabilizes a vector $v\in\mathfrak{V}$ is $\mathfrak{b}_{Mv}^{H,J}$ is a multiple of $\mathfrak{b}_{v}^{H,J}$ (or pairs non-trivially with it, which is the same by this corollary). By setting $\delta_{H,J}(M,v)$ to be the constant satisfying $\mathfrak{b}_{Mv}^{H,J}=\mathfrak{b}_{v}^{H,J}$, the argument from the proof of Theorem \ref{hypplane}, but now using Theorem \ref{arrowJperp} and Remark \ref{Jperpiso}, yields \[\dim\big[\big(\uparrow_{J}\mathbb{C}[B]\big)^{\perp}\big]^{\mathrm{inv}}=\sum_{v\in\mathfrak{V}}\sum_{M\in\operatorname{St}(v)}\frac{\delta_{H,J}(M,v)\varepsilon_{J}(M,v)\mathbf{e}\big(Q(M,v)\big)}{p^{3}-p},\] with our modified notion of the stabilizer.

Moreover, a conjugation argument implies that if $M\in\operatorname{St}(v)$, $\alpha$ is the constant such that $\rho_{D}(M)\mathfrak{b}_{v}^{H,J}=\alpha\mathfrak{b}_{v}^{H,J}$, and $N\in\operatorname{SL}_{2}(\mathbb{F}_{p})$, then the element $NMN^{-1}$ stabilizes $Nv$ and we have $\rho_{D}(NMM^{-1})\mathfrak{b}_{Nv}^{H,J}=\alpha\mathfrak{b}_{Nv}^{H,J}$ with the same constant $\alpha$. Therefore we can once again replace $\mathfrak{V}$ by a subset $\mathfrak{O}$ consisting of one representative for each orbit of $\operatorname{SL}_{2}(\mathbb{F}_{p})$ on $\mathfrak{V}$ modulo the relations from Corollary \ref{basisTl}, and obtain
\begin{equation}
\dim\big[\big(\uparrow_{J}\mathbb{C}[B]\big)^{\perp}\big]^{\mathrm{inv}}=\sum_{v\in\mathfrak{O}}\sum_{M\in\operatorname{St}(v)}\frac{\delta_{H,J}(M,v)\varepsilon_{J}(M,v)\mathbf{e}\big(Q(M,v)\big)}{|\operatorname{St}(v)|}. \label{dimwithO}
\end{equation}

Now, the images in $D/H^{\perp}$ of the entries $\eta$ and $\lambda$ of a vector $v\in\mathfrak{O}\subseteq\mathfrak{V}$ span a subspace that is not contained in $J^{\perp}/H^{\perp}$, and can therefore be of dimension 1 or 2. Moreover, this subspace remains invariant under the action of $\operatorname{SL}_{2}(\mathbb{F}_{p})$, and thus for each such subspace we can consider the set of orbit associated with it. Via the action of $\operatorname{SL}_{2}(\mathbb{F}_{p})$ we may assume that in any element $v=\binom{\eta}{\lambda}\in\mathfrak{O}$ we have $\lambda \in J^{\perp}$ and $\eta \not\in J^{\perp}$, so that $\mathfrak{b}_{v}^{H,J}=\mathfrak{a}_{v}^{H}$. Thus Corollary \ref{basisTl} (or Lemma \ref{ortho}) implies that for a representative associated with a subspace of $D/H^{\perp}$, the $\eta$-coordinate only appears in the basis via its $D/H^{\perp}$-image, but there are $p^{r}$ different $\lambda$'s in the basis with the same given $D/H^{\perp}$-image.

Given a representative $v\in\mathfrak{O}$ that is associated with a 2-dimensional subspace of $D/H^{\perp}$, the group $\operatorname{St}(v)$ is trivial, and thus such an orbit contributes 1 to the right hand side of Equation \eqref{dimwithO}. Moreover, as in Lemma \ref{SLnongen}, the group $\operatorname{SL}_{2}(\mathbb{F}_{p})$ acts on the set of bases for such a space with $\varphi(p)=p-1$ orbits. Since there are $\frac{(p^{d}-1)(p^{d-1}-1)}{(p^{2}-1)(p-1)}$ 2-dimensional subspaces of $D/H^{\perp}$, out of which $\frac{(p^{d-1}-1)(p^{d-2}-1)}{(p^{2}-1)(p-1)}$ are contained in $J^{\perp}/H^{\perp}$ and must be excluded, and each subspace is associated with $p^{r}(p-1)$ orbits, these combine to a total of $p^{r}\big[\frac{(p^{d}-1)(p^{d-1}-1)}{p^{2}-1}-\frac{(p^{d-1}-1)(p^{d-2}-1)}{p^{2}-1}\big]$ (this can be simplified to $p^{d+r-2}(p^{d-1}-1)$, but the expanded form is better for merging with the induction hypothesis).

On the other hand, for $v\in\mathfrak{O}$ for which the subspace of $D/H^{\perp}$ is 1-dimensional, we have $\lambda \in H^{\perp}$ and therefore the powers of $T$ stabilize $v$, with $\delta_{H,J}(T^{l},v)=\varepsilon_{J}(T^{l},v)=1$ (see Proposition \ref{formbHJ} and Corollary \ref{basisTl} or Lemma \ref{ortho} once again), and $Q(T^{l},v)=l\frac{\lambda^{2}}{2}$ from Equation \eqref{QMv}. It follows that the contribution of such an orbit is $\frac{1}{p}\sum_{l\in\mathbb{F}_{p}}\mathbf{e}\big(l\frac{\lambda^{2}}{2}\big)$, which is 1 if $\lambda$ is isotropic and 0 otherwise. But since we assumed that $H$ was a maximal isotropic subgroup of $D$, the only contributing elements are those with $\lambda \in H$. Therefore for each one of the $\frac{p^{d}-1}{p-1}$ 1-dimensional subspaces of $D/H^{\perp}$, except for the $\frac{p^{d-1}-1}{p-1}$ subspaces that are contained in $J^{\perp}/H^{\perp}$, we have a contribution of 1 (this difference reduces to $p^{d-1}$, but once again we leave it in the expanded form).

Adding these terms to the formula for $\dim\mathbb{C}[B]^{\mathrm{inv}}$ from the induction hypothesis gives the desired result. This proves the theorem.
\end{proof}
Note that the case $r=0$ in Theorem \ref{Fpvs} is also covered by Theorem \ref{hypplane}, essentially with $G=D/H^{\perp}=D/H$. As the calculations from the proof show, the three terms in Theorem \ref{Fpvs} are $S_{p,p}(G)(p-1)$, $S_{p,1}(G)$, and $S_{1,1}(G)$ respectively. On the other hand, with $r>0$ we only have $2^{-(2d+2)}_{II}$ with $r=2$ when $p=2$, and for odd $p$ the possibilities are $p^{\pm(2d+1)}$ with $r=1$ or $p^{\varepsilon(2d+2)}$ with signature 4 and $r=2$.

Finally, note that when $d=1$ and $r>0$ Theorem \ref{Fpvs} gives $\dim\mathbb{C}[D]^{\mathrm{inv}}=1$, so it is interesting to see what form does a generator of this space takes. It turns out that for $r=2$, i.e., when $D$ is $p^{-4}$ or $2^{-4}_{II}$, this vector has a simple form $\sum_{0\neq\gamma \in D_{\mathrm{iso}}}\mathfrak{e}_{\gamma}-(p-1)\mathfrak{e}_{0}$. For odd primes $p$ and $r=1$, one can show that while the special orthogonal group of $D$ as a quadratic space over $\mathbb{F}_{p}$ operates transitively on the non-zero isotropic vectors in $D$, its connected component (defined by trivial spinor norm) operates with two orbits. Moreover, for such an isotropic element $\gamma$, a multiple $l\gamma$ of $\gamma$ lie in the same orbit as $\gamma$ if and only if $\binom{l}{p}=+1$. Writing these orbits as $\mathcal{O}_{\pm}$ (with an arbitrary choice of signs), the expression $\sum_{\gamma\in\mathcal{O}_{+}}\mathfrak{e}_{\gamma}-\sum_{\gamma\in\mathcal{O}_{-}}\mathfrak{e}_{\gamma}$ spans $\mathbb{C}[D]^{\mathrm{inv}}$ (the invariance under $\rho_{D}(S)$ and the form of intersections with isotropic lines is closely related to the classical Gauss sum for $p$).

For $d>1$, the space $\mathbb{C}[D]^{\mathrm{inv}}$ will contain the 1-dimensional space $\uparrow_{K}\mathbb{C}[C]$ for each the isotropic subgroups $K$ of cardinality $p^{d-1}$ in $D$. It is an interesting question whether these vectors generate $\mathbb{C}[D]^{\mathrm{inv}}$ in this case, much analogously to the case $r=0$ already considered in \cite{[NRS]} and \cite{[Bi]}. We leave the question of which types of vectors generate the space of invariants for more general discriminant forms for future research. Another question, which is also left for future research, is to determine, for some discriminant forms $D$, the subspace of $\mathbb{C}[D]$ on which $\rho_{D}$ operates via a given non-trivial character of $\operatorname{Mp}_{2}(\mathbb{Z})$. The method proving Theorems \ref{hypplane} and \ref{Fpvs} can surely be applied to shed light on this question as well at some instances.

\medskip

\noindent\textsc{Einstein Institute of Mathematics, the Hebrew University of Jerusalem, Edmund Safra Campus, Jerusalem 91904, Israel}

\noindent E-mail address: zemels@math.huji.ac.il

\end{document}